\documentclass{amsart}

\usepackage[utf8]{inputenc}
\usepackage{amsmath,amsfonts}
\usepackage{amsthm}
\usepackage{amssymb}
\usepackage{graphicx}
\usepackage{enumitem}
\usepackage{genyoungtabtikz}
\usepackage{ytableau}

\allowdisplaybreaks

\def\sa{\vskip .125truein}

\def\allowhooks{tunnel hooks}
\def\allowhook{tunnel hook}
\def\allowhookfilling{tunnel hook covering}
\def\allowhookfillings{tunnel hook coverings}
\def\allowcell{tunnel cell}
\def\allowcells{tunnel cells}

\def\bordercells{boundary cells}

\def\bordercell{boundary cell}

\def\THF{THC}
\def\THC{THC}
\def\GBPR{GBPR}

\def\murmin{\mu}

\def\h{\mathfrak{h}}

\def\c{\tau}
\def\n{\xi}
\def\K{\mathbf{K}}
\def\HS{H}
\def\tc{\c_r}
\def\td{t}

\newtheorem{thm}{Theorem}
\newtheorem{ex}[thm]{Example}
\newtheorem{Note}[thm]{Note}
\newtheorem{lem}[thm]{Lemma}
\newtheorem{cor}[thm]{Corollary}
\newtheorem{prop}[thm]{Proposition}
\newtheorem{proc}[thm]{Procedure}
\newtheorem{defn}[thm]{Definition}
\newtheorem{theorem}[thm]{Theorem}

\numberwithin{thm}{section}

\definecolor{lightblue}{rgb}{.85,.9,.96}
\definecolor{lightpurple}{rgb}{.9,.5,.9}
\definecolor{lightred}{rgb}{.9,.7,.7}
\definecolor{grey}{rgb}{.85,.85,.85}

\def\sign{\epsilon}
\def\calN{\mathcal{N}}
\def\calT{\mathcal{T}}
\def\calB{\mathcal{B}}

\Yboxdimx{30pt}
\Ylinecolour{black}
\Ynodecolour{black!50!black}

\usepackage[colorinlistoftodos]{todonotes}

\DeclareMathOperator{\I}{\mathfrak{S}}
\DeclareMathOperator{\dI}{\mathfrak{S}^{\star}}
\DeclareMathOperator{\Sym}{Sym}
\DeclareMathOperator{\SSYT}{SSYT}
\def\SYM{\Sym}
\DeclareMathOperator{\NSym}{NSym}
\def\NSYM{\NSym}
\DeclareMathOperator{\QSym}{QSym}

\DeclareMathOperator{\spin}{bank}
\DeclareMathOperator{\munu}{\mu/\nu}

\DeclareMathOperator{\munurplus}{\mu/\nu^{(r)}}
\DeclareMathOperator{\SHF}{THC}
\DeclareMathOperator{\taxi}{taxi}
\DeclareMathOperator{\aug}{aug}
\DeclareMathOperator{\fl}{flat}
\DeclareMathOperator{\ndet}{\mathfrak{det}}

\DeclareMathSymbol{\shortminus}{\mathbin}{AMSa}{"39}

\begin{document}

\title[Comb. NSym inverse Kostka] {A combinatorial interpretation of the noncommutative inverse Kostka matrix}


\author{Edward E. Allen}
\address{Dept. of Mathematics, Wake Forest University, Winston-Salem, NC 27109}
\curraddr{}
\email{allene@wfu.edu}
\thanks{}

\author{Sarah K. Mason}
\address{Dept. of Mathematics, Wake Forest University, Winston-Salem, NC 27109}
\curraddr{}
\email{masonsk@wfu.edu}
\thanks{SM was partially supported by an AMS-Simons PUI grant.}


\date{}

\dedicatory{}

\begin{abstract}
   We provide a combinatorial formula for the expansion of immaculate noncommutative symmetric functions into
    complete homogeneous noncommutative symmetric functions.  To do this, we introduce generalizations of Ferrers diagrams which we call \GBPR~diagrams.
A \GBPR~diagram assigns a color (grey, blue, purple, or red) to each cell
of the diagram.
We define \allowhooks, which play a role similar to that of the special rim hooks appearing in the E\u{g}ecio\u{g}lu-Remmel formula for the symmetric inverse Kostka matrix.  We extend this interpretation to skew shapes and fully generalize to define immaculate functions indexed by integer sequences skewed by integer sequences.  Finally, as an application of our combinatorial formula, we extend Campbell's results on ribbon decompositions of immaculate functions to a larger class of shapes.
\end{abstract}

\maketitle

\section{Introduction}

The ring $\Sym$ of symmetric functions on a set of commuting variables consists of all polynomials invariant under the action of the symmetric group.  Symmetric functions play an important role in representation theory, combinatorics, and other areas of mathematics and the physical and natural sciences.  Bases for $\Sym$ are indexed by partitions; two ubiquitous examples are the \emph{Schur functions} $s_{\lambda}$ and the \emph{complete homogeneous symmetric functions} $h_{\lambda}$.  Schur functions correspond to irreducible representations of the symmetric group, and their multiplication corresponds to the cohomology of the Grassmannian~\cite{Ful97,Mac95,Sag01}.

The inverse Kostka matrix is the transition matrix from the Schur basis of $\Sym$ to the complete
homogeneous basis.  Objects called \emph{special rim hooks} are used by E\u{g}ecio\u{g}lu and Remmel to construct a combinatorial interpretation of this matrix~\cite{EgeRem90}, originating from the Jacobi-Trudi formula.

The Hopf algebra $\NSym$ of noncommutative symmetric functions is freely generated by a collection of noncommutative
algebraically independent generators $\HS_i$, one at each positive degree $i$.  The set $\{ H_{\alpha}:= H_{\alpha_1} H_{\alpha_2} \cdots H_{\alpha_k} \}$, indexed by compositions $\alpha$ of $n$, forms a basis for $\NSym$.  The map $\chi: \NSym \rightarrow \Sym$ defined by $\chi(H_{\alpha})=h_{\alpha_1} h_{\alpha_2} \cdots h_{\alpha_k}$ (sometimes called the ``forgetful map since it ``forgets" that the generators don't commute) sends elements of $\NSym$ to elements of $\Sym$.

The immaculate basis for $\NSym$ is a Schur-like basis that maps to the Schur functions under the forgetful map, introduced in~\cite{BBSSZ14} through creation operators analogous to those used to construct the Schur functions.  It is natural to ask if there is a generalization to $\NSym$ of the combinatorial inverse Kostka formula, since there is a Jacobi-Trudi style formulation for the immaculate basis in terms of the complete homogeneous basis for $\NSym$~\cite{BBSSZ14}.  
In this paper, we answer this question in the affirmative, providing a combinatorial formula for the expansion of the immaculate basis into the complete homogeneous basis for $\NSym$ using ribbon-like objects we call \emph{\allowhooks}.  
Specifically, we prove the following theorem.

\begin{theorem}
The decomposition of the immaculate noncommutative symmetric functions into the
complete homogeneous noncommutative symmetric functions is given by the following formula.
\begin{equation}
\I_{\mu}=\sum_{\gamma\in \SHF_{\mu}} \prod_{r=1
}^k \ \sign(\h(r,\c_r))\ H_{\Delta(\h(r,\c_r))},
\label{Eq:TH11}
\end{equation}
where $\mu\in\mathbb{Z}^k,$
 $\SHF_{\mu}$ denotes the collection of \allowhookfillings~of a diagram of shape $\mu,$ and a sign $\sign(\h(r,\c_r))$ together with an integer value $\Delta(\h(r,\c_r))$ are assigned to each \allowhook~$\h(r,\c_r)$ in each $\gamma \in \SHF_{\mu}$.
\label{T:main}
\end{theorem}

\noindent
Note that the product 
$\prod_{r=1
}^k \ \sign(\h(r,\c_r))\ H_{\Delta(\h(r,\c_r))}$ 
in Equation \eqref{Eq:TH11} is taken in order from $r=1$ to $k$ so that 
$$\prod_{r=1
}^k \ \sign(\h(r,\c_r))\ H_{\Delta(\h(r,\c_r))}= \sign(\h(1,\c_1))\ H_{\Delta(\h(1,\c_1))} \cdots \sign(\h(k,\c_k))\ H_{\Delta(\h(k,\c_k))},
$$
since the functions $H_{\Delta(\h(r,\c_r))}$ do not commute.  

The process of constructing \allowhookfillings~is an iterative process which introduces a skew-shape generalization along the way.  This naturally leads to the introduction of a candidate for the \emph{skew immaculate noncommutative symmetric functions} $\I_{\mu/\lambda}$ and a generalization of Theorem~\ref{T:main} to the skew setting.  Although this construction is compatible with skew Schur functions (following~\cite{Mac95}) under the forgetful map, this doesn't always align with the Hopf algebraic skewing operator.  See Section~\ref{S:modifynu} for details.

\begin{thm}{\label{T:mainskew}}
The decomposition of the skew immaculate noncommutative symmetric functions $\I_{\mu/\lambda}$ (with $\mu \in\mathbb{Z}^k$ and $\lambda$ a partition with at most $k$ parts) into the
complete homogeneous noncommutative symmetric functions is given by the following formula.
\begin{equation}
\I_{\mu/\lambda}=\sum_{\gamma\in \SHF_{\mu/\nu}} \prod_{r=1
}^k \ \ \sign(\h(r,\c_r))\ H_{\Delta(\h(r,\c_r))},
\end{equation}
where $\SHF_{\mu/\lambda}$ denotes the collection of \allowhook~coverings of a diagram of shape $\mu/\lambda$, and a sign $\sign(\h(r,\c_r))$ together with an integer value $\Delta(\h(r,\c_r))$ are assigned to each \allowhook~$\h(r,\c_r)$ in each $\gamma \in \SHF_{\mu/\lambda}$. 
\end{thm}

As in Theorem~\ref{T:main}, the product is taken in order from $r=1$ to $k$ since the order of the functions matters in $\NSym$.

Loehr and Niese recently published a combinatorial interpretation of the nonsymmetric inverse Kostka matrix~\cite{LoeNie21}.  Their approach uses transitive tournaments and recursively defined sums, which is quite different from our computationally expedient diagrammatic approach.  

In addition to the tournament approach, Loehr and Niese also provide a diagrammatic method for computing the decomposition of an immaculate into the complete homogeneous basis when the indexing shape is a partition. Our diagrammatic approach works for all indexing shapes, including all compositions and also all sequences of integers.  Our decomposition can be determined directly by looking at the diagram and recording the values of the \allowhooks.

In Section 2 we review important definitions and properties concerning the rings $\SYM$, $\NSYM$, and $\QSym$.  In Section 3 we generalize Ferrers diagrams to provide what we call \emph{\GBPR~diagrams} for sequences of integers skewed by partitions.  We also define 
\allowhooks~and \allowhookfillings, the main objects involved in our combinatorial formulas.
Section 4 contains the proof of Theorem~\ref{T:main}.  
In Section \ref{S:modifynu}
we use determinants of submatrices to generalize Theorem~\ref{T:main} to immaculates indexed by sequences skewed by partitions.  We then extend this further to sequences skewed by arbitrary sequences and
describe the relation between special rim hooks \cite{EgeRem90} and \allowhooks.
In Section \ref{sec:ribbon}
we apply Theorem \ref{T:main}
to extend results of Campbell
\cite{Cam17} regarding the ribbon decompositions
of immaculates.

\section{Foundation and background}

We begin by reviewing necessary material to motivate our work.  There are a number of excellent sources~\cite{BBSSZ14,Cam17,GKLLRT95,Mac95,Sag01} for further background in this area.

\subsection{Classical combinatorial notions}

A \emph{sequence} $\mu$ of length $k$ (denoted $\ell(k)$)
consists of a $k$-tuple of integers
$(\mu_1,\mu_2,\ldots,\mu_k)$
where $\mu_i\in\mathbb{Z}$
for $1 \le i \le k.$
A \emph{weak composition} is a sequence whose entries are all non-negative integers.  A \emph{strong composition} 
 is a weak composition all of whose positive entries appear before any zeros.  Often, we simply use the term \emph{composition} to refer to a strong composition.  A \emph{partition} $\lambda=(\lambda_1,\lambda_2,\ldots,\lambda_k)$
is a composition in which $0 \le \lambda_{i+1} \le \lambda_i$
for $1 \le i \le k-1.$ 
 Write $\mu \models n$ to denote that $\mu$ is a composition of $n$ and $\lambda \vdash n$ to denote that $\lambda$ is a partition of $n$.  Partitions are often represented by  Ferrers diagrams in which
(using French notation) there are $\lambda_i$ boxes (called \emph{cells}) in row $i$
going from the bottom (south) to the top (north).

A \emph{skew shape} is a pair of partitions $\mu/\lambda$ such that $\lambda_i \le \mu_i$ for $1 \le i \le k$. 
The \emph{diagram} of a skew shape $\mu/\lambda$, where $\mu=(\mu_1,\mu_2,\ldots,\mu_k)$ and $\lambda=(\lambda_1,\lambda_2, \hdots , \lambda_k)$,
has
$\mu_i$ cells in row $i$ with 
the left-most $\lambda_i$ cells shaded out.  A cell in row $i$ and column $j$ is indicated by $(i,j)$.  A partition or weak composition $\mu$ is identified with the skew shape $\mu/(0,0, \hdots , 0)$.

\begin{ex}
The diagram of shape $${\mu/\lambda}=(4,3,2,2,1)/(2,2,1,0,0)$$
is given by the labelled cells (row, column) in the following figure.
\vskip .125truein

\hskip 1.46truein
\vbox{
\Yfillcolour{white}
\young({{{{5,1}}}})
\vskip -1.295pt
\young({4,1}{4,2})
\vskip -1.295pt
\Yfillcolour{grey}
\young(\ !<\Yfillcolour{white}>{3,2})
\vskip -1.295pt
\Yfillcolour{grey}
\young(\ \ !<\Yfillcolour{white}>{2,3})
\vskip -1.295pt
\Yfillcolour{grey}
\young(\ \ !<\Yfillcolour{white}>{1,3}{1,4})
\label{Ex:FerrersDiagram}
}
\end{ex}

\begin{Note}
Generally speaking, throughout this paper, 
$\mu$ and $\nu$ denote sequences (which, depending on the context, includes compositions and weak compositions),
while $\lambda$ is strictly used for partitions.
\end{Note}
 
A \emph{semi-standard Young tableau} ($\SSYT$) of shape $\mu/\lambda$ (where $\mu$ and $\lambda$ are partitions) is a filling of the non-shaded cells of the diagram with positive integers, so that the entries weakly increase along rows from left to right (west to east) and strictly increase up columns from bottom to top (south to north).  The set of all semi-standard Young tableaux of shape $\mu/\lambda$ is denoted by $\SSYT(\mu/\lambda)$.  The \emph{weight} $x^T$ of a semi-standard Young tableau $T$ is given by $$x^T = \prod x_i^{\textrm{ \# of times $i$ appears in $T$}}.$$  In the following definition, we suppress the variable set, which can either be finite $\{x_1, x_2, \hdots , x_{\ell}\}$, in which case $\SSYT(\mu/\lambda)$ includes all semi-standard Young tableaux with entries in $\{1,2, \hdots , \ell\}$, or infinite $\{x_1, x_2, \hdots\}$, allowing for all positive integers to appear as entries in the semi-standard Young tableaux in $\SSYT(\mu/\lambda)$.

\begin{defn}
The \emph{Schur function} $s_{\mu/\lambda}$ is defined by $$s_{\mu/\lambda}=\sum_{T \in \SSYT(\mu/\lambda)} x^T.$$
\end{defn}

\subsection{The symmetric group and symmetric functions}


A permutation $\sigma= (\sigma_1, \sigma_2 \hdots , \sigma_n) \in S_n$, written in one-line notation where $S_n$ denotes the symmetric group, acts on a formal power series $f(x_1, x_2, \hdots , x_n)$ in $n$ variables by $$\sigma (f(x_1, \hdots , x_n)) = f(x_{\sigma_1}, x_{\sigma_2}, \hdots , x_{\sigma_n}).$$
The formal power series $f(x_1,x_2,\ldots, x_n)$ in $n$ variables
is said to be symmetric
if
$$\sigma(f(x_1,x_2,\ldots, x_n))
=f(x_1,x_2,\ldots, x_n)
$$
for all $\sigma\in S_n.$
The ring of symmetric functions in $n$ variables,
denoted by $\Sym_n$, consists of all symmetric formal
power series $f(x_1, x_2, \hdots, x_n)$ on commuting variables $x_1, x_2,\hdots , x_n$ with coefficients from a field $\K$.  This notion can be extended to symmetric functions in infinitely many variables; a formal power series $f(x_1, x_2, \hdots)$ is in $\Sym$ if $$\sigma (f(x_1, x_2, \hdots) )= f(x_1, x_2, \hdots)$$ for every permutation $\sigma$ of the positive integers.  Bases for $\Sym$ are indexed by partitions.

Schur functions form an important basis for $\Sym$ since they correspond to characters of irreducible representations of the symmetric group and also to the cohomology of the Grassmannian~\cite{Ful97}.  Skew Schur functions generalize the Schur functions to pairs of indexing partitions while enjoying many properties similar to those of the Schur functions indexed by partitions~\cite{Mac95}.

Two other useful bases for $\Sym$ are the \emph{complete homogeneous symmetric functions $h_{\lambda}$} and the \emph{elementary symmetric functions} $e_{\lambda}$ where $\lambda=(\lambda_1,\ldots,\lambda_k)$ is a partition.  
One way to define $h_{\lambda}$ is to set $h_n=s_n$ (for $n \in \mathbb{Z}^+$) and $h_{\lambda} = h_{\lambda_1} h_{\lambda_2} \cdots h_{\lambda_k}$.
Similarly, set $e_n=s_{1^n}$ (for $n \in \mathbb{Z}^+$) and $e_{\lambda} = e_{\lambda_1} e_{\lambda_2} \cdots e_{\lambda_k}$ for $\lambda=(\lambda_1 , \lambda_2 , \hdots, \lambda_k)$.  The \emph{Jacobi-Trudi Formula} provides a rule for expanding the Schur 
functions $s_{\lambda/\nu}$ into the complete homogeneous symmetric functions.

\begin{theorem}[Jacobi-Trudi Formula]{\cite{Mac95}}
Let $\lambda/\nu$ be a skew shape.  Then 
\begin{equation}
s_{\lambda/\nu}=\det(h_{\lambda_i-i-(\nu_j-j)})_{i,j}
\label{E:SymJacobiTrudi}\end{equation}\label{T:SymJacobiTrudi}
\end{theorem}

The inverse Kostka matrix is the transition matrix from the Schur function basis to the complete homogeneous symmetric function basis.   E\u{g}ecio\u{g}lu and Remmel's combinatorial interpretation of the inverse Kostka matrix~\cite{EgeRem90} provides a method for writing Schur functions in terms of complete homogeneous symmetric functions using decompositions of the indexing shapes into collections of cells called \emph{special rim hooks}.  A \emph{rim hook tabloid} is constructed by repeated removal of special rim hooks from the diagram of a partition.  Each rim hook tabloid has an associated type and a sign, and these are used to determine the coefficients appearing in the expansion of Schur functions into the complete homogeneous symmetric functions.  (See Section~\ref{sec:schur} for further details.)

\subsection{Quasisymmetric  and noncommutative symmetric functions}

Quasisymmetric functions generalize symmetric functions.  They first appear in the work of Stanley~\cite{Sta72} and are formally developed by Gessel~\cite{Ges84}.
A polynomial (or more generally a bounded degree formal power series on an infinite alphabet $x_1, x_2, \hdots$) $f$ is \emph{quasisymmetric} if the coefficient of 
$x_1^{\alpha_1} x_2^{\alpha_2} \cdots x_k^{\alpha_k}$ in $f$ is equal to the coefficient of $x_{j_1}^{\alpha_{1}} x_{j_2}^{\alpha_{2}} \cdots x_{j_k}^{\alpha_{k}}$ 
in $f$ for any sequence $j_1< j_2 < \hdots < j_k$ of positive integers and any composition $(\alpha_1, \hdots, \alpha_k)$.
Quasisymmetric functions form the terminal object in the category of combinatorial Hopf algebras~\cite{ABS06}.

Bases for the algebra $\QSym$ of quasisymmetric functions 
over a field $\K$
are indexed by compositions  
$\alpha=(\alpha_1, \alpha_2, \hdots, \alpha_k)$. 
The \emph{monomial quasisymmetric function} $M_{\alpha}$ is obtained by taking the monomial $x^{\alpha}=x_1^{\alpha_1}\cdots x_k^{\alpha_k}$ and ``quasisymmetrizing" it.  That is, $$M_{\alpha} = \sum_{j_1 < j_2 < \cdots < j_k} x_{j_1}^{\alpha_1} x_{j_2}^{\alpha_2} \cdots x_{j_k}^{\alpha_k}.$$

The \emph{fundamental quasisymmetric functions}
are a basis for $\QSym$ which can be defined in terms of monomial quasisymmetric functions.  In particular, $$F_{\alpha} = \sum_{\beta \succeq \alpha} M_{\beta}$$
where $\beta=(\beta_1,\beta_2,\ldots,\beta_j)$,
$\alpha=(\alpha_1,\alpha_2,\ldots,\alpha_k)$,
and $\succeq$ is the refinement order.  (Recall
$\beta \succeq \alpha$ means that there exists a sequence
$(g_1,\ldots,g_j)$ with $g_m<g_{m+1}$ for $1 \le m \le j-1$ such that
$\beta_i=\alpha_{g_{i-1}+1}+\cdots+\alpha_{g_{i}}$
for $1 \le i \le j$ with $g_0=0.$)

Each Schur function decomposes into a positive sum of fundamental quasisymmetric functions based on a statistic called the \emph{descent set}~\cite{Sta99}.  Fundamental quasisymmetric functions correspond to characters of irreducible representations of the $0$-Hecke algebra~\cite{DKLT96}.


The Hopf algebra $\NSym$ is the dual of $\QSym$~\cite{GKLLRT95}.  Recall $\NSym$ can be thought of as the free associative algebra $\K \langle H_1, H_2 , \hdots \rangle$ generated by algebraically independent, noncommuting complete homogeneneous symmetric functions $H_i$ over a fixed commutative field $\K$ of characteristic zero.  Set $H_a:=0$ if $a$ is a negative integer and $H_0:=1$. 
 Then $H_{(\alpha_1, \alpha_2, \hdots, \alpha_{k})} = H_{\alpha_1}H_{\alpha_2} \cdots H_{\alpha_{k}}$ for $\alpha_i\in\mathbb{Z}.$  The collection of complete homogeneous noncommutative symmetric functions indexed by compositions forms a basis for $\NSym$. 

\subsection{Creation operators and immaculate functions}

The Schur function basis for $\Sym$ can be defined using creation operators.  Let ${\bf B}_m: \Sym_n \rightarrow \Sym_{m+n}$ be the \emph{Bernstein operator} defined by $${\bf B}_m := \sum_{i \ge 0} (-1)^i h_{m+i} e_i^{\perp},$$ in which $e_i^\perp$ is the adjoint operator defined by $$\langle e_i g, h \rangle = \langle g, e_i^{\perp} h \rangle \; \; \textrm{ for all } g,h \in \Sym,$$ where $\langle,\rangle$ is the inner product on $\Sym$ for which the Schur functions are self-dual.

For any sequence $\mu \in \mathbb{Z}^k$, the Schur function $s_{\mu}$ can be constructed by repeated application of Bernstein operators \cite{Zel81}; 
specifically,
$$s_{\mu} = {\bf B}_{\mu_1}{\bf B}_{\mu_2}   \cdots {\bf B}_{\mu_k}(1).$$
These creation operators can be extended to $\NSym$ by replacing $h_{m+i}$ with the complete homogeneous non-commutative symmetric function $H_{m+i}$ and replacing $e_i^{\perp}$ with the linear transformation $F_{1^i}^{\perp}$ of $\NSym$ that is adjoint to multiplication by the fundamental $F_{1^i}$ in $\QSym$.  This construction leads to the following definition.

\begin{defn}{\cite{BBSSZ14}}  The \emph{noncommutative Bernstein operator} $\mathbb{B}_m$ is given by
$$\mathbb{B}_m=\sum_{i \ge 0} (-1)^i H_{m+i} F_{1^i}^{\perp}.$$
\end{defn}

The immaculate functions $\I_\mu$ are then defined analogously to the Schur functions by the following application of creation operators.

\begin{defn}{\cite{BBSSZ14}}  The \emph{immaculate noncommutative symmetric function $\I_{\mu}$} is defined by
$$\I_\mu=\mathbb{B}_{\mu_1} \cdots \mathbb{B}_{\mu_k}(1).$$
\label{D:Immacdef}
\end{defn}

\noindent Immaculate functions correspond to indecomposable modules of the $0$-Hecke algebra~\cite{BBSSZ15}.  The following $\NSym$ version of the Jacobi-Trudi Theorem provides a determinantal formula for the expansion of the immaculate functions $\I_\mu$ into the complete homogeneous basis for $\NSym$.

\begin{theorem}[\cite{BBSSZ14}]
For an arbitrary integer sequence $\mu$, let $M_\mu$ be the matrix given by $(M_{\mu})_{i,j}=H_{\mu_i+j-i}$.  Then
$\I_\mu=\ndet(M_\mu).$
\label{T:NsymJacobiTrudi}
\end{theorem}

\noindent
The above formula for $\I_\mu$ requires that
the determinant be computed using Laplace expansion starting with the top row and then continuing down to the bottom row.
We will use $\ndet(M_\mu)$ instead of $\det(M_\mu)$
to represent this strict expansion process.

\begin{ex}
If $\mu=(-1,3,2)$, then $$M_{\mu} = \begin{bmatrix} 
H_{-1} & H_0 & H_1  \\ 
H_{2} & H_{3} & H_4  \\ 
H_{0} & H_{1} & H_2  
\end{bmatrix}.$$  
With $H_a=0$ if $a \in \mathbb{Z}^{-}$ and $H_0=1$, the decomposition of $\I_{\mu}$ into the complete homogeneous basis for $\NSym$ is given by $$\I_{\mu}=H_{(4)}-H_{(2,2)}+H_{(1,2,1)} - H_{(1,3)}.$$
\end{ex}




\section{\GBPR~diagrams and \allowcells}{\label{Sec:gbpr}}

Skew diagrams, as traditionally defined, do not extend to the full generality necessary for our work.  
For our diagrams, we will need to consider shapes $\mu/\nu$ where $\mu$ is an integer
sequence (which in particular might contain negative parts) and $\nu$ is a partition.  Allowing $\mu$ to have negative parts requires modifying the typical diagram visualization.

For all of Section~\ref{Sec:gbpr}, let $\mu=(\mu_1,\ldots,\mu_k)$ be an integer
sequence  and
$\nu=(\nu_1,\ldots,\nu_k)$ be a partition.
The \emph{Grey-Blue-Purple-Red (\GBPR) diagram} $D_{\munu}$ for the skew shape ${\munu}$ is obtained via the following process. 

\begin{enumerate}
\item Place $\nu_i$ grey cells in row $i$ of the diagram (for $1 \le i \le k$), working from bottom to top (to stay consistent with French notation).

\item For each $1 \le i \le k$, there are several cases that impact the placement of red and blue cells into row $i$.
\begin{enumerate}
\item If $\mu_i >0$ and $\nu_i \le \mu_i,$ place $\mu_i-\nu_i$ blue cells in row $i$ situated immediately to the right of the grey cells. 

\item If $\mu_i>0$ and $\mu_i < \nu_i,$ place $\nu_i-\mu_i$ red cells in row $i$ situated immediately to the right of the grey cells.

\item If $\mu_i \le 0,$ place $|\mu_i|+\nu_i$ red cells  in row $i$ situated immediately to the right of the grey cells.
\end{enumerate}

\item Any cell in the first quadrant that is not colored grey, red, or blue is \textit{purple}, but we do not typically color these in illustrations since there are infinitely many purple cells.  Purple cells are considered part of the diagram and are available to be claimed (converted to grey or red) as necessary in the following applications.
Often, we will draw a purple cell 
in cell location $(i,\nu_i+1)$ 
when the row does not have any blue or red cells just to be clear 
that the the cell is available for claiming.
\end{enumerate}

Although the \GBPR~diagram does not necessarily contain exactly $\mu_i$ cells in row $i$, the value of $\mu_i$ can be determined from the number of grey, blue, and red cells in the following way.

\begin{lem}\label{lem:gbprsum}
If row $i$ of a \GBPR~diagram has
$a_i$ grey cells, $b_i$ blue cells,  
and $c_i$ red cells, then $a_i+b_i-c_i=\mu_i.$
\end{lem}

\begin{proof}
To see this, consider each of the cases in Step 2 of the \GBPR~diagram construction.  In all cases, $a_i=\nu_i$.

In case (a), $b_i=\mu_i-\nu_i$ and $c_i=0$, so $$a_i+b_i-c_i=\nu_i+\mu_i-\nu_i-0=\mu_i.$$

In cases (b) and (c), 
$b_i=0$ and $c_i=\nu_i-\mu_i$, so $$a_i+b_i-c_i=\nu_i+0-(\nu_i-\mu_i)=\mu_i.$$  Therefore in all cases, $\mu_i=a_i+b_i-c_i$.
\end{proof}

Let $C_{\munu}$ denote the collection of red and blue cells in a diagram $D_{\munu}$.
 
\begin{ex}
The following is the diagram $D_{\munu}$ where \\
$\mu=(-3,1,-1,0,3,-2)$ and 
$\nu=(2,2,1,0,0,0)$ with the corresponding location indicated in each cell, respectively.
\sa
\hskip .72truein
\vbox{
\Yfillcolour{lightred}
\young({6,1}{6,2},)
\vskip -1.295pt
\Yfillcolour{lightblue}
\young({5,1}{5,2}{5,3},)
\vskip -1.295pt
\Yfillcolour{lightpurple}
\young({4,1},)
\vskip -1.295pt
\Yfillcolour{grey}
\young({3,1}!<\Yfillcolour{lightred}>{3,2}{3,3},)
\vskip -1.295pt
\Yfillcolour{grey}
\young({2,1}{2,2}!<\Yfillcolour{lightred}>{2,3},)
\vskip -1.295pt
\Yfillcolour{grey}
\young({1,1}{1,2}!<\Yfillcolour{lightred}>{1,3}{1,4}{1,5}{1,6}{1,7})
\label{Ex:negfour}
}

\sa
\noindent
Note
$$C_{\nu}=\{(1,1),(1,2),(2,1),(2,2),(3,1)\}.$$ 
and
$$C_{\munu}=\{(1,3),(1,4),(1,5),(1,6),(1,7),(2,3),(3,2),(3,3),
(5,1), $$ $$(5,2), (5,3),(6,1),(6,2)\}.$$ 
\end{ex}

\begin{defn}{\label{def:connected}}
Two cells $(p,q)$ and $(r,s)$ are said to be
adjacent if and only if $|p-r|+|q-s|=1.$
We say that a collection $C$ of cells is \emph{connected} if for any cells $c, d \in C$
there is a sequence $c=\c_1, \c_2, \ldots, \c_j=d$ where $\c_i\in C$ and $\c_i$ and $\c_{i+1}$ are adjacent for $1 \le i \le j-1.$  A set consisting of just one cell is also considered to be connected.  Cells $(p,q)$ and $(r,s)$ are \emph{diagonally adjoining} if both $r=p+1$ and $s=q+1$ or both $p=r+1$ and $q=s+1$.
\end{defn}

At times, it will be convenient to construct diagrams for shapes obtained via prefix removal.  In particular, at times we encounter a pair of sequences $\mu \in \mathbb{Z}^k$, and $\nu \in \mathbb{Z}_{\ge 0}^k$ such that $\nu_{r+1} \ge \nu_{r+2} \ge \hdots \ge \nu_k$.  If the first $r$ parts of $\nu$ are not weakly decreasing, we can remove them and consider parts $r+1$ through $k$ of $\mu$ and $\nu$, which allows us to still construct a \GBPR~diagram for the remaining rows.  The following definition makes this precise. 

\begin{defn}
Let $r\in\mathbb{Z}_{\ge 0}, \mu \in \mathbb{Z}^k$, and $\nu \in \mathbb{Z}_{\ge 0}^k$ such that $\nu_{r+1} \ge \nu_{r+2} \ge \hdots \ge \nu_k$.  The \emph{partial diagram} $D^{(r)}_{\mu/\nu}$ is obtained by constructing the \GBPR~diagram $D_{(\mu_{r+1}, \mu_{r+2} , \hdots , \mu_k)/(\nu_{r+1}, \nu_{r+2} , \hdots,  \nu_k)}$ and then shifting the resulting diagram up by $r$ rows, so that the first nonempty row is in row $r+1$ of the plane.
\end{defn}

The row labels for a shifted diagram correspond to their placement in the $xy$-plane.  Note that any \GBPR~diagram can be considered as a partial diagram by setting $r=0$; that is, $$D_{\mu/\nu} = D^{(0)}_{\mu/\nu}.$$

\begin{ex}
With $D_{\mu/\nu}$
given in Example \ref{Ex:negfour},
$D_{\mu/\nu}^{(2)}$ is given by

\hskip 1.75truein
\vbox{
\Yfillcolour{lightred}
\young({6,1}{6,2},)
\vskip -1.295pt
\Yfillcolour{lightblue}
\young({5,1}{5,2}{5,3},)
\vskip -1.295pt
\Yfillcolour{lightpurple}
\young({4,1},)
\vskip -1.295pt
\Yfillcolour{grey}
\young({3,1}!<\Yfillcolour{lightred}>{3,2}{3,3},)
}
\end{ex}

We now define \emph{\bordercells}, which intuitively are the cells which lie on the boundary of $\nu$.  In other words, a boundary cell is either horizontally or vertically adjacent to at least one cell in $\nu$, or diagonally adjoining a cell in $\nu$ (as in Definition \ref{def:connected}).  Any red or blue cell in row $r+1$ of the plane in the partial diagram $D^{(r)}_{\mu/\nu}$ is also a boundary cell.  
Boundary cells will be used to construct hooks that will play a role similar to the role of \emph{rim hooks} in the combinatorial interpretation of the symmetric inverse Kostka matrix~\cite{EgeRem90}.

\begin{defn}
\label{Def:Bordcells}
Let $\mu \in \mathbb{Z}^k$ and let $r$ be a positive integer such that $1 \le r \le k$.  Let $\nu \in \mathbb{Z}_{\ge 0}^k$ such that $(\nu_r, \nu_{r+1}, \hdots , \nu_k)$ is a partition.  A cell in location $(p,q)$ (with $r \le p \le k$) 
is a \emph{\bordercell}~of $D_{\munu}^{(r-1)}$ iff $$\nu_p+1 \le q\le\nu_{p-1}+1 \; \; \textrm{ (for } p>r)$$
and 
$$\nu_r+1 \le q\le\max\{\nu_r+1, a_r+b_r+c_r\} \; \; \textrm{ (if } p=r).$$
Recall $a_r,$ $b_r$ and $c_r$ are the number of grey cells, blue cells and red cells in row $r$, respectively. 
A \emph{\allowcell} $(p,q)$ of $D^{(r-1)}_{\munu}$
is a \bordercell~such that $q=\nu_p+1.$  
\end{defn}


\noindent 
Note that \bordercells~are not necessarily
cells of $C_{\munu}$ (i.e., \bordercells~may
be red, blue, or purple).  The inequality 
$\nu_r+1\le q\le\max\{\nu_r+1, a_r+b_r+c_r\}$ for row $r$ forces the cell $(r,\nu_r+1)$, as well as all red or blue cells in row $r$, to be boundary cells.  Also note that a cell $(p,q)$ with $p<r$ cannot be a boundary cell of $D^{(r)}_{\mu/\nu}$.

Let $\mathcal{B}_{\munu}^{(r)}$ 
and $\mathcal{T}_{\munu}^{(r)}$
denote the collections of \bordercells~and \allowcells~of $D_{\munu}^{(r)}$, respectively.  Let $\mathcal{N}_{\munu}^{(r)}$ be the set difference $$\mathcal{N}_{\munu}^{(r)}:=\mathcal{B}_{\munu}^{(r)}
\setminus\mathcal{T}_{\munu}^{(r)}.$$ 

\begin{defn}
\label{D:tunhookdef}
Let $\c=(p,q) \in \mathcal{T}_{\munu}^{(r-1)}$ be a \allowcell.  A \allowhook~on $D_{\munu}^{(r-1)}$ is a collection $\h(r,\c)$ consisting of all boundary cells in rows $r$ through $p$. 
 The cell $\c$ (the farthest northwest cell of $\h(r,\c)$) is called the \textit{terminal cell} of $\h(r,\c)$.  The \emph{sign} of \allowhook~$\h(r,\c)$, denoted by $\sign(\h(r,\c)),$ equals $(-1)^{p-r}$.  If cell
$\hat c\in \h(r,\c)$ then we say that $\h(r,\c)$ \emph{covers}  $\hat c$.
\end{defn}

There are a number of hook/strip objects in the literature such as skew hooks, ribbons, and border strips~\cite{Mac95}, and the rim hooks~\cite{EgeRem90} appearing in the combinatorial interpretation of the inverse Kostka matrix in $\Sym$.  In order to generalize these combinatorial objects to the $\NSym$ setting, \allowhooks~need to ``tunnel" into the diagram instead of remaining on the rim.  See Section~\ref{sec:schur} for a detailed comparison of rim hooks and \allowhooks.

\begin{ex}
Let $\mu=(5,4,-4,3,-2,5,3)$, $\nu=(5,4,2,2,2,1,0)$, and $r=2$.  Since $(\nu_3,\nu_4, \hdots, \nu_7)$ is a partition, we have the following \GBPR~diagram $D^{(2)}_{\mu/\nu}$ in which we have selected a \allowhook~$$\h(3,(6,2))=\{(3,8),(3,7),(3,6),(3,5),(3,4),(3,3),$$
$$(4,3),(5,3),(6,3),(6,2)\}.$$

\vskip -.125truein
\hskip .45truein
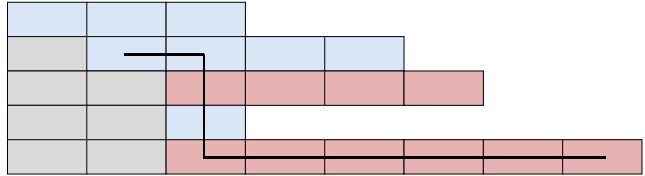
\begin{picture}(300,75)
\put(0,0){
\vbox{
\Yfillcolour{grey}
\young(
!<\Yfillcolour{lightblue}>\ \ \    )
\vskip -1.295pt
\young(\ !<\Yfillcolour{lightblue}>\ \ \ \ )
\vskip -1.295pt
\Yfillcolour{grey}
\young(\  \ !<\Yfillcolour{lightred}>\ \ \ \ )
\vskip -1.295pt
\Yfillcolour{grey}
\young(\  \ !<\Yfillcolour{lightblue}>\   )
\vskip -1.295pt
\Yfillcolour{grey}
\young(\  \ !<\Yfillcolour{lightred}>\ \ \ \ \ \ )
}
}


\put(90,3){\line(1,0){152}}
\put(90,3){\line(0,1){39}}
\put(60,42){\line(1,0){30}}

\end{picture}

\vskip 5pt

\noindent
The \bordercells~in the above diagram are $$\mathcal{B}_{\munu}^{(2)}=\{(3,3),(3,4),(3,5),(3,6),(3,7),(3,8),$$
$$(4,3),(5,3),(6,2),(6,3),(7,1),(7,2)\},$$
the \allowcells~are
$$\mathcal{T}_{\munu}^{(2)}=\{(3,3),(4,3),(5,3),(6,2),(7,1)\},$$
and $$\mathcal{N}_{\munu}^{(2)}=\{(3,4),(3,5),(3,6),(3,7),(3,8),(6,3),(7,2) \}.$$
\label{Ex:munu}
\end{ex}

The following lemmas will be useful in later proofs. 

\begin{lem}{\label{lem:bdrydiag}}
If $(p,q) \in \mathcal{B}^{(r)}_{\munu}$, then $(p+1,q+1) \notin \mathcal{B}^{(r)}_{\munu}$.
\end{lem}

\begin{proof}
Let $(p,q) \in \mathcal{B}^{(r)}_{\munu}$.  Then $\nu_p + 1 \le q$ by Definition \ref{Def:Bordcells}.  Adding $1$ to both sides implies that $\nu_p+2 \le q+1$.  Then $\nu_p+1 < \nu_p+2 \le q+1$ contradicts the inequality $q+1 \le \nu_p+1$ necessary for $(p+1,q+1) \in \mathcal{B}^{(r)}_{\munu}$.  Therefore, $(p+1,q+1) \notin \mathcal{B}^{(r)}_{\munu}$ and the proof is complete.
\end{proof}

Note that Lemma~\ref{lem:bdrydiag} implies the set of all boundary cells does not contain any $2 \times2$ rectangles.  Another consequence of Lemma~\ref{lem:bdrydiag} is that for any cell $(p,q)$ contained in a \allowhook~in the diagram $D^{(r)}_{\munu}$, the cell $(p+1,q+1)$ is not a \bordercell~in $D^{(r)}_{\munu}$.  The following lemma shows that every \allowhook~is connected.

\begin{lem}{\label{lem:takesall}}
Let $\h(r,\c)$ be a \allowhook~in $D^{(r)}_{\munu}$ terminating at cell $\c=(p,q)$.  Then $\h(r,\c)$ is connected.
\end{lem}

\begin{proof}
For $r<i < p$, the boundary cells $Y_{i}$ and $Y_{i+1}$ of rows $i$ and $i+1$ are, respectively,
$$Y_i=\{(i,\nu_{i}+1),(i,\nu_{i}+2),\ldots,(i,\nu_{i-1}+1)\}$$ and
$$Y_{i+1}=\{(i+1,\nu_{i+1}+1),(i+1,\nu_{i+1}+2),\ldots,(i+1,\nu_{i}+1)\}.$$
  Each of these collections is connected.  Since $(i,\nu_{i}+1)$ and $(i+1,\nu_{i}+1)$ are connected, the union $Y_{i} \cup Y_{i+1}$ is connected, and hence the collection $\h(r,\c)$ is connected as long as the collection $Y_r$ of boundary cells in row $r$ is connected to the boundary cells $Y_{r+1}$ in row $r+1$.  Since $(r,\nu_r+1) \in Y_r$ and $(r+1, \nu_r+1) \in Y_{r+1}$, the union $Y_r \cup Y_{r+1}$ is connected and the proof is complete.
\end{proof}

\begin{defn}{\label{D:bank}}
Given a row $i$ of $D^{(r)}_{\munu}$,
define
\begin{align}
    \spin_{\munu}(i)=&(\text{\# of blue cells in row $i$) --
    (\text{\# of red cells in row $i$})}\notag
\end{align}
\end{defn}
\sa

\noindent
With $\h(r,\c)$ a \allowhook~in the partial diagram $D^{(r)}_{\munu}$ terminating at cell $\c=(p,q)$,
set
\begin{equation}
\Delta(\h(r,\c))=\spin_{\munu}(r) + (\nu_r+1-q)+(p-r).
\label{E:Delta}
\end{equation}
Note that $(\nu_r+1-q)+(p-r)$ is the \emph{taxicab}
(or \emph{Manhattan}) distance from the cell $(r,\nu_r+1)$
to the cell $\c=(p,q)$ \cite{kra86}.
Therefore, set 
\begin{equation}
\taxi(\h(r,\c))= (\nu_r+1-q)+(p-r).
\label{Def:taxi}
\end{equation}
In Equation~\eqref{E:Delta}, since $\nu$ is a partition (whose parts might be $0$),
the cell $(p,q) \in {\calT}^{(r)}_{\munu}$ will be weakly west of 
$(r,\nu_r+1)$.

\begin{ex}
Consider the partial diagram $D^{(2)}_{\munu}$ depicted in Example \ref{Ex:munu}.  The \allowhook~shown has
$\Delta(\h(3,(6,2)))=-6+4=-2,$
since there are six red cells in row $3$ and  $\taxi(\h(3,(6,2)))=4.$  Similarly (but not shown), $\Delta(\h(3,(7,1)))=-6+6=0$, $\Delta(\h(3,(5,3)))=-6+2=-4$, $\Delta(\h(3,(4,3)))=-6+1=-5,$ and $\Delta(\h(3,(3,3)))=-6+0=-6.$
\label{Ex:Deltas}
\end{ex}

The proof of the following lemma is immediate since there is only one \allowcell~in each row and moving up a row strictly increases the value of $\Delta(\h(r,\c))$.  

\begin{lem}{\label{T:UniqueLengths}}
Given a partial diagram $D^{(r)}_{\mu/\nu}$ with $\mu$ a sequence and $\nu$ a partition, for any fixed $j \in \mathbb{Z}$, there is at most one \allowcell~$\c$ such that $j=\Delta(\h(r,\c))$.
\end{lem}

The following iterative procedure provides a method for constructing a {\emph \allowhookfilling~($\SHF$)} of the diagram $D_{\munu}$.  

\begin{proc}{\label{A:THC}}
Consider a sequence $\mu=(\mu_1, \mu_2, \hdots , \mu_k) \in \mathbb{Z}^k$ and a partition $\nu=(\nu_1, \nu_2 , \hdots , \nu_k)$.

\begin{enumerate}

\item Construct the partial \GBPR~diagram $D^{(0)}_{\munu}$ of shape $\munu$.  Set $\nu^{(0)}=\nu$.

\item Repeat the following steps, once for each value of $r$ from $1$ to $k$.

\begin{enumerate}

\item Choose a \allowhook~$\h(r,\c_r)$ in $D^{(r-1)}_{\mu/\nu^{(r-1)}}$ and set $$\alpha_r:=\Delta(\h(r,\c_r)).$$

\item{\label{St:grey}} For each $1 \le i \le k$, let $\eta_i^{(r)}$ be the number of cells in row $i$ of $\h(r,\c_r)$ and let $\nu^{(r)}$ be the sequence defined for $1 \le i \le k$ by$$\nu_{i}^{(r)}=\nu_{i}^{(r-1)}+\eta_{i}^{(r)}.$$

\item{\label{St:gbpr}} Construct the partial \GBPR~diagram $D^{(r)}_{\mu/\nu^{(r)}}$.  

\end{enumerate}
\item Let $\gamma$ denote the resulting \allowhookfilling~and set \newline $\Delta(\gamma) = (\alpha_1, \hdots , \alpha_k).$ 
\end{enumerate}
\end{proc}

Step~\ref{St:grey} appends the cells covered by $\h(r,\c)$ to $\nu^{(r-1)}$.  Although $\nu^{(r)}$ is not necessarily a partition, we will prove in Lemma~\ref{L:partition} that $\nu^{(r)}_{r+1} \ge \nu^{(r)}_{r+2} \ge \nu^{(r)}_{k}$, as is necessary in order to construct the partial diagram $D^{(r)}_{\mu/\nu^{(r)}}$. 

In the diagram $D^{(r)}_{\mu/\nu^{(r)}}$ constructed in Step~\ref{St:gbpr}, the newly appended cells become grey.  Additional red cells are appended to the remaining rows in Step~\ref{St:gbpr}. One additional red cell is appended to row $i$ for every purple cell covered by $\h(r,\c)$ in row $i$ and two additional red cells are appended to row $i$ for every red cell covered by $\h(r,\c)$ in row $i$.

\begin{lem}{\label{L:partition}} 
Let $\mu \in \mathbb{Z}^k$ be an integer sequence and $\nu^{(r)}$ be a sequence of nonnegative integers produced during Step~\ref{St:grey} of Procedure~\ref{A:THC}.  Then $(\nu^{(r)}_{r+1}, \nu^{(r)}_{r+2} , \hdots , \nu^{(r)}_k)$ is a partition.
\end{lem}

\begin{proof}
First note that for $r=0$, the sequence $(\nu^{(0)}_1, \nu^{(0)}_2, \hdots , \nu^{(0)}_k) = (0,0, \hdots , 0)$ and therefore $(\nu^{(0)}_1, \nu^{(0)}_2, \hdots , \nu^{(0)}_k)$ is a partition.  
Assume that $(\nu_{j+1}^{(j)}, \nu_{j+2}^{(j)}, \hdots , \nu_{k}^{(j)})$ is a partition for $0 \le j \le r-1$ and let $\nu^{(r)}$ be a sequence of nonnegative integers produced during Step~\ref{St:grey} of Procedure~\ref{A:THC} for the partial \GBPR~diagram $D^{(r-1)}_{\mu/\nu^{(r-1)}}$.  Let $\h(r,\c_r)$ be the \allowhook~constructed during this process.  We will prove that $\nu^{(r)}_{i} \ge \nu^{(r)}_{i+1}$ for $r+1 \le i < k$.

First note that since $(\nu_{r}^{(r-1)}, \nu_{r+1}^{(r-1)}, \hdots , \nu_{k}^{(r-1)})$ is a partition by assumption, $\nu^{(r-1)}_{i} \ge \nu^{(r-1)}_{i+1}$ for $r \le i < k$.  Assume $r<i < k$ and notice this means we don't need to consider cells from row $r$, which eliminates the need to consider the second inequality in the boundary cell definition (Definition~\ref{Def:Bordcells}).  If $\h(r,\c_r)$ does not include cells from row $i+1$, then $\nu^{(r)}_{i+1} = \nu^{(r-1)}_{i+1} \le  \nu^{(r-1)}_i  \le \nu^{(r)}_i$ so that $\nu^{(r)}_{i+1} \le \nu^{(r)}_i,$ as desired.

If $\h(r,\c_r)$ includes cells from row $i+1$, it must also include cells from row $i$.  Recall the definition of \allowhooks~states that every boundary cell from an included row must be contained in $\h(r,\c_r)$.  The largest column index for a boundary cell in row $i+1$ is $\nu^{(r-1)}_{i+1-1}+1 = \nu^{(r-1)}_i+1$, whereas the largest column index for a boundary cell in row $i$ is $\nu^{(r-1)}_{i-1}+1$.  Since $\nu^{(r-1)}_{i}+1 \le \nu^{(r-1)}_{i-1}+1 $, the largest column index for a boundary cell in row $i+1$ is weakly smaller than the largest column index for a boundary cell in row $i$.  Therefore, $\nu^{(r)}_{i+1} \le \nu^{(r)}_i$, and the proof is complete.
\end{proof}

\begin{ex}
We describe a \allowhookfilling~of $D_{\munu}$ when 
$\munu=(-3,5,5,\\ 0,5,-2,4,6)/(2,1)$.  
First, we give the \GBPR~diagram of $D^{(0)}_{\munu}$.
We then use a table to provide details for a particular \allowhookfilling. 
Finally, we illustrate this particular $\THF$ on the \GBPR~diagram itself.  
\vskip .125truein

The \GBPR~diagram of $D^{(0)}_{\munu}$ is shown below, with each cell labelled by its corresponding row and column.
\vskip .125truein

\hskip .675truein
\hbox{
\vbox{
\Yfillcolour{lightblue}
\young({8,1}{8,2}{8,3}{8,4}{8,5}{8,6},)
\vskip -1.295pt
\Yfillcolour{lightblue}
\young({7,1}{7,2}{7,3}{7,4},)
\vskip -1.295pt
\Yfillcolour{lightred}
\young({6,1}{6,2},)
\vskip -1.295pt
\Yfillcolour{lightblue}
\young({5,1}{5,2}{5,3}{5,4}{5,5},)
\vskip -1.295pt
\Yfillcolour{lightpurple}
\young({4,1},)
\vskip -1.295pt
\Yfillcolour{lightblue}
\young({3,1}{3,2}{3,3}{3,4}{3,5},)
\vskip -1.295pt
\Yfillcolour{grey}
\young({2,1}!<\Yfillcolour{lightblue}>{2,2}{2,3}{2,4}{2,5},)
\vskip -1.295pt
\Yfillcolour{grey}
\young({1,1}{1,2}!<\Yfillcolour{lightred}>{1,3}{1,4}{1,5}{1,6}{1,7})
}}
\label{Fig:mu0nu0start}

\vskip .125truein
\noindent
The following table records the process of decomposing the \GBPR~diagram $D_{\munu}$ into \allowhooks.  
Each row $r$ indicates the situation before the $r^{th}$ \allowhook~is placed.  Here $\c_r$ is the \allowcell~at which the \allowhook~beginning in the $r^{th}$ row of the partial diagram terminates.  Notice that although $\nu^{(r-1)}$ is not always a partition, the last $r$ terms of $\nu^{(r-1)}$ do form a partition, so that $D^{(r-1)}_{\mu/\nu^{(r-1)}}$ is a valid partial \GBPR~diagram.  

\begin{center}
\begin{tabular}{|c|c|c|c|c|}
\hline
$r$ & $(\mu_r,\mu_{r+1}, \hdots , \mu_k)$  & $\nu^{(r-1)}$ &   $\c_r$ &  
$\Delta(\h(r, \c_r))$\\
\hline
$1$ & $(-3,5,5,0,5,-2,4,6)$ & $(2,1,0,0,0,0,0,0)$ & $(5,1)$  & $-5+6=1$   \\
\hline
$2$ & $(5,5,0,5,-2,4,6)$ & $(7,3,2,1,1,0,0,0)$ & $(2,4)$   & $2+0=2$  \\
\hline
$3$ & $(5,0,5,-2,4,6)$ & $(7,5,2,1,1,0,0,0)$ & $(4,2)$   & $3+2=5$  \\
\hline
$4$ & $(0,5,-2,4,6)$ & $(7,5,5,3,1,0,0,0)$ & $(5,2)$   & $-3+3=0$ \\
\hline
$5$ & $(5,-2,4,6)$ & $(7,5,5,6,4,0,0,0)$ & $(5,5)$  & $1+0=1$ \\
\hline
$6$ & $(-2,4,6)$ & $(7,5,5,6,5,0,0,0)$  & $(8,1)$  & $-2+2=0$ \\
\hline
$7$ & $(4,6)$ & $(7,5,5,6,5,2,1,1)$  & $(8,2)$   & $3+1=4$ \\
\hline
$8$ & $(6)$ & $(7,5,5,6,5,2,4,2)$  & $(8,3)$   & $4+0=4$\\ 
\hline
\end{tabular}
\end{center}

\noindent
Next, we illustrate $\gamma$, the $\THF$ of $D^{(0)}_{\munu}$ indicated in the previous table.  We leave all the rows in the diagram so that we can depict all the \allowhooks~at once.  We also omit the step of converting the colors of the cells to grey as they are covered by \allowhooks, so that their color as they are covered by a tunnel hook is retained.

\vskip .125truein
\hskip .35truein
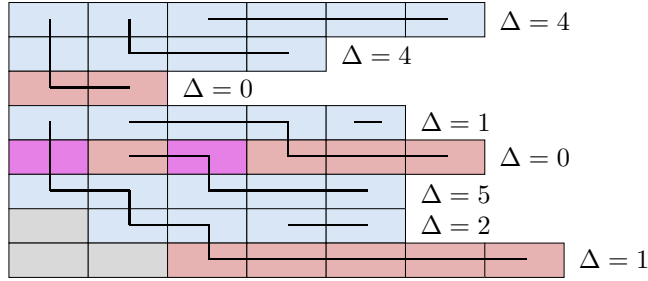
\begin{picture}(100,100)
\put(0,0){
\hbox{
\vbox{
\Yfillcolour{lightblue}
\young(\ \ \ \ \ \ ) 
\vskip -1.295pt
\Yfillcolour{lightblue}
\young(\ \ \ \ )
\vskip -1.295pt
\Yfillcolour{lightred}
\young(\ \ )
\vskip -1.295pt
\Yfillcolour{lightblue}
\young(\ \ \ \ \ )
\vskip -1.295pt
\Yfillcolour{lightpurple}
\young(\  !<\Yfillcolour{lightred}>\ !<\Yfillcolour{lightpurple}>\ !<\Yfillcolour{lightred}>\ \ \ )
\vskip -1.295pt
\Yfillcolour{lightblue}
\young(\  \ \ \ \ )
\vskip -1.295pt
\Yfillcolour{grey}
\young(\ !<\Yfillcolour{lightblue}>\ \ \ \ )
\vskip -1.295pt
\Yfillcolour{grey}
\young(\ \ !<\Yfillcolour{lightred}>\ \ \ \ \ )

}}
}
\put(95,3.5){\line(1,0){120}}
\put(95,3.5){\line(0,1){13}}
\put(65,16.5){\line(1,0){30}}
\put(65,16.5){\line(0,1){13}}
\put(35,29.5){\line(1,0){30}}
\put(35,29.5){\line(0,1){26}}
\put(235,0){$\Delta=1$}
\put(125,16.5){\line(1,0){30}}
\put(175,13){$\Delta=2$}
\put(95,29.5){\line(1,0){60}}
\put(95,29.5){\line(0,1){13}}
\put(65,42.5){\line(1,0){30}}
\put(175,25){$\Delta=5$}
\put(125,42.5){\line(1,0){60}}
\put(125,42.5){\line(0,1){13}}
\put(65,55.5){\line(1,0){60}}
\put(205,39){$\Delta=0$}
\put(150,55.5){\line(1,0){10}}
\put(175,52){$\Delta=1$}
\put(35,68.5){\line(1,0){30}}
\put(35,68.5){\line(0,1){26}}
\put(85,65){$\Delta=0$}
\put(65,81.5){\line(1,0){60}}
\put(65,81.5){\line(0,1){13}}
\put(145,77.5){$\Delta=4$}
\put(95,94.5){\line(1,0){90}}
\put(205,90.5){$\Delta=4$}


\end{picture}

\vspace*{.1in}

\noindent
Therefore $\Delta(\gamma)=(1,2,5,0,1,0,4,4)$.
\end{ex}

\section{Theorem \ref{T:main} proof and application to $\QSym$}

In this section, we prove Theorem~\ref{T:main}.  Recall Theorem~\ref{T:NsymJacobiTrudi}~\cite{BBSSZ14}
states that 
$$\I_\mu=\ndet(M_\mu)=\sum_{\sigma\in S_k}
\sign(\sigma) (M_\mu)_{i,\sigma_i}.
$$  Theorem~\ref{T:main} provides a combinatorial interpretation of this decomposition of immaculate functions into complete homogeneous functions using diagram fillings called~\allowhookfillings.  The main idea behind our proof
is a natural bijection between \allowhookfillings~and permutations
(Lemma~\ref{T:perm}).
Furthermore, each \allowhookfilling~is associated to a number which we will show is equal to the subscript of the corresponding complete homogeneous function appearing in the matrix $M_{\mu}$ (Lemma~\ref{T:Mrj}).  
We complete the proof by attaching signs to tunnel hooks and showing that the product of the signs of the \allowhooks~in a \allowhookfilling~is equal to the sign of the corresponding permutation.  We close the section with an application to $\QSym$, obtaining a combinatorial formula for the expansion of monomial quasisymmetric functions in terms of dual immaculates as a corollary to Theorem~\ref{T:main}.

\subsection{Transitions from ${\calT}^{(r)}_{\mu/\nu^{(r)}}$
to
${\calN}^{(r)}_{\mu/\nu^{(r)}}$
}

The first lemma necessary for the proof of Theorem~\ref{T:main} provides foundational insight into how \allowcells~are related to the diagonals parallel to the line $y=x$.  Let
\begin{equation}
\mathcal{L}_{j}=\{(p,q) | p-q+1=j\} = \{(j+m,1+m) | m \in \mathbb{Z}^{\ge 0} \}
\label{E:LA}
\end{equation}
be the collection of cells in the $j^{th}$ diagonal of the first quadrant of the plane, for $1 \le j \le k$.  These diagonals (whose properties are described in the following lemma) will correspond to the entries in the permutations used when computing the determinant of the matrix $M_{\mu}$.
Recall the definition of
border cells $\mathcal{B}^{(r)}_{\mu/\nu^{(r)}}$,
tunnel cells 
${\calT}^{(r)}_{\mu/\nu^{(r)}}$,
and ${\calN}^{(r)}_{\mu/\nu^{(r)}}$
found in 
and immediately following Definition~\ref{Def:Bordcells}.

\begin{lem}
\label{T:Losecells}
Let $D^{(r-1)}_{\mu/\nu^{(r-1)}}$ be a partial \GBPR~diagram for $\mu \in \mathbb{Z}^k$ and $\nu^{(r-1)} \in \mathbb{Z}^k_{\ge 0}$ such that $(\nu^{(r-1)}_r, \nu^{(r-1)}_{r+1}, \hdots , \nu^{(r-1)}_k)$ is a partition.  Assume $r<k$.  Suppose $\tc,\td \in{\calT}^{(r-1)}_{\mu/\nu^{(r-1)}}$, $\tc \ne \td$, and
$\n\in{\calN}^{(r-1)}_{\mu/\nu^{(r-1)}}.$ Furthermore, suppose $\tc =(p_1,q_1),$
$\td =(p_2,q_2),$ and $\n=(p_3,q_3)$.  Finally,
let $\h(r,\tc)$ be a \allowhook~in 
the diagram $D^{(r-1)}_{\mu/\nu^{(r-1)}}.$
Then
\begin{enumerate}[label=\Alph*.]
    \item $(p_1+1,q_1+1)\in {\calN}^{(r)}_{\mu/\nu^{(r)}}.$
    
    \item If $\h(r,\tc)$ does not cover $\td$ then
     $\td\in{\calT}^{(r)}_{\mu/\nu^{(r)}}$.
     
     \item If $\h(r,\tc)$ covers $\td$ then $(p_2+1,q_2+1)\in{\calT}^{(r)}_{\mu/\nu^{(r)}}.$
     
     \item
     If $\h(r,\tc)$ does not cover $\n$ then $\n\in{\calN}^{(r)}_{\mu/\nu^{(r)}}.$
    
     \item
     If $\h(r,\tc)$ covers $\n$ then either $(p_3+1,q_3+1)\in{\calN}^{(r)}_{\mu/\nu^{(r)}}$  or $p_3=k$ in which case $(p_3+1,q_3+1) \notin D^{(r)}_{\mu/\nu^{(r)}}$.
    
\end{enumerate}
\end{lem}
\noindent

\begin{proof}[Proof of part A]

First note that $p_1, p_2, p_3 \ge r$ since the diagram begins in row $r$.  Since $\tc=(p_1,q_1)$ is the terminal cell of
$\h(r,\tc)$, we have
$\nu_{p_1}^{(r-1)} = q_1-1,$ $\nu_{p_1}^{(r)} \ge q_1,$
and $\nu_{p_1+1}^{(r-1)}=\nu_{p_1+1}^{(r)}$.  Also, the fact that $(\nu^{(r-1)}_r, \nu^{(r-1)}_{r+1}, \hdots , \nu^{(r-1)}_k)$ is a partition (Lemma~\ref{L:partition}) implies
$$\nu_{p_1+1}^{(r)} = \nu_{p_1+1}^{(r-1)}\le  \nu_{p_1}^{(r-1)}=q_1-1.$$
Thus, $$\nu_{p_1+1}^{(r)} + 1 \le q_1 < q_1+1 \le \nu_{p_1}^{(r)}+1,$$ so that $(p_1+1,q_1+1) \in \mathcal{B}^{(r)}_{\mu/\nu^{(r)}}$.
However, $(p_1+1,q_1+1)\notin \mathcal{T}^{(r)}_{\mu/\nu^{(r)}}$ since $\nu_{p_1+1}^{(r)} \le q_1-1$ implies $\nu_{p_1+1}^{(r)} +1 \not= q_1+1$.  Therefore, $(p_1+1,q_1+1)\in \mathcal{N}^{(r)}_{\mu/\nu^{(r)}}$, recalling that $\mathcal{N}^{(r)}_{\mu/\nu^{(r)}}$ is the set difference $\mathcal{B}^{(r)}_{\mu/\nu^{(r)}} \setminus \mathcal{T}^{(r)}_{\mu/\nu^{(r)}}$.
\end{proof}

\begin{proof}[Proof of part B]
If $\h(r,\tc)$ does not cover 
$\td=(p_2,q_2),$ then $p_2>p_1$ by the \allowhook~definition.  Therefore,
$q_2=\nu_{p_2}^{(r-1)}+1=\nu_{p_2}^{(r)}+1$
so that $\td\in{\calT}^{(r)}_{\mu/\nu^{(r)}}$.
\end{proof}

\begin{proof}[Proof of part C]
If $\h(r,\tc)$ covers
$\td=(p_2,q_2)\in {\calT}^{(r-1)}_{\mu/\nu^{(r-1)}},$ then
$p_2<p_1$ since $\tc \not= \td$
and at most one cell of each row of
$D^{(r-1)}_{\mu/\nu^{(r-1)}}$
belongs to 
${\calT}^{(r-1)}_{\mu/\nu^{(r-1)}}$.
Therefore, $\h(r,\tc)$ includes all boundary cells from row $p_2+1$.  The \allowhook~$\h(r,\tc)$ does not cover
$(p_2+1,q_2+1)$ since $(p_2+1,q_2+1) \notin \mathcal{B}^{(r-1)}_{\mu/\nu^{(r-1)}}$ by Lemma~\ref{lem:bdrydiag}.  Since $\td \in {\calT}^{(r-1)}_{\mu/\nu^{(r-1)}}$, we have
$\nu_{p_2}^{(r-1)}={q_2-1}.$   This means $(p_2, q_2-1) \notin \h(r,\tc)$, so that all cells in row $p_2$ covered by $\h(r,\tc)$ lie weakly to the right of $(p_2,q_2)$. 
 The fact that $q_2+1 > \nu_{p_2}^{(r-1)}$ implies the cells in row $p_2+1$ covered by $\h(r,\tc)$ must lie strictly to the left of $(p_2+1, q_2+1)$.  Since $\h(r,\c_r)$ is connected by Lemma \ref{lem:takesall} and includes cells in row $p_2+1$ (since $p_2+1 \le p_1$), the cell $(p_2+1,q_2)$ must be covered by $\h(r,\tc)$.  Since $(p_2+1,q_2)$ is covered by $\h(r, \tc)$ but $(p_2+1,q_2+1)$ is not covered by $\h(r,\tc),$ we have $q_2+1 = \nu_{p_2+1}^{(r)} +1$ and $(p_2+1,q_2+1) \in \mathcal{T}^{(r)}_{\mu/\nu^{(r)}}$.
\end{proof}

\begin{proof}
[Proof of part D]
Let $\n=(p_3,q_3)$ be a \bordercell~which is not a \allowcell~
in $D^{(r-1)}_{\mu/\nu^{(r-1)}}$; 
i.e., $\n \in {\calN}^{(r-1)}_{\mu/\nu^{(r-1)}}$.  Assume $\h(r,\tc)$ does not cover $\n$.
Therefore, $p_1<p_3$ and $\nu_{p_3}^{(r)}=\nu_{p_3}^{(r-1)}$.  So
$\n$ is a boundary cell in $D^{(r)}_{\mu/\nu^{(r)}}$ but not  a \allowcell~in $D^{(r)}_{\mu/\nu^{(r)}}$ since otherwise $\n$ would be a \allowcell~in $D^{(r-1)}_{\mu/\nu^{(r-1)}}$.  Therefore, $\n \in {\calN}^{(r)}_{\mu/\nu^{(r)}}.$
\end{proof}

\begin{proof}
[Proof of part E]
Let $\n=(p_3,q_3)$ be a \bordercell~which is not a \allowcell; 
i.e., $\n \in {\calN}^{(r-1)}_{\mu/\nu^{(r-1)}}$, and assume $\h(r,\tc)$ covers $\n$.  Assume $p_3<k$; otherwise $(p_3+1,q_3+1) \notin D^{(r)}_{\mu/\nu^{(r)}}$.  We first prove $(p_3+1,q_3+1) \in {\calB}^{(r)}_{\mu/\nu^{(r)}}$ and then show $(p_3+1,q_3+1) \notin {\calT}^{(r)}_{\mu/\nu^{(r)}}.$  

The assumption that $\h(r,\tc)$ covers $\n$ implies $q_3 \le \nu^{(r)}_{p_3}$.  Adding $1$ to each side of this inequality gives 
\begin{equation}{\label{eq:rhsbd}}
    q_3+1 \le \nu^{(r)}_{p_3}+1.
\end{equation}
Since $(p_3,q_3) \in {\calB}^{(r-1)}_{\mu/\nu^{(r-1)}},$ Lemma~\ref{lem:bdrydiag} implies that $(p_3+1,q_3+1) \notin {\calB}^{(r-1)}_{\mu/\nu^{(r-1)}}.$  
Thus, $(p_3+1,q_3+1)$ is not covered by $\h(r,\tc)$, since $\h(r,\tc)$ only covers boundary cells.  
Therefore, $\nu^{(r)}_{p_3+1} <q_3+1$ and thus 
\begin{equation}{\label{eq:lhsbd}}
\nu^{(r)}_{p_3+1}+1 \le q_3+1.
\end{equation}
Together Equations~\eqref{eq:rhsbd} and~(\ref{eq:lhsbd}) imply that 
$$\nu^{(r)}_{p_3+1}+1 \le q_3+1 \le \nu^{(r)}_{p_3}+1,$$ meaning $(p_3+1,q_3+1) \in {\calB}^{(r)}_{\mu/\nu^{(r)}}$.

Now we show $(p_3+1,q_3+1) \notin {\calT}^{(r)}_{\munurplus}$, by proving $q_3+1 \not= \nu^{(r)}_{p_3+1}+1$.  Since $(p_3,q_3) \notin {\calT}^{(r-1)}_{\murmin/\nu^{(r-1)}},$ the fact that $(p_3,q_3)$ is covered by $\h(r,\tc)$ also implies that $(p_3,q_3-1)$ is covered by $\h(r,\tc)$ by the \allowhook~definition.  So $(p_3,q_3-1) \in {\calB}^{(r-1)}_{\murmin/\nu^{(r-1)}}$.  Therefore, we have $(p_3+1, q_3) \notin {\calB}^{(r-1)}_{\murmin/\nu^{(r-1)}}$ by Lemma~\ref{lem:bdrydiag}.  Then, $\nu_{p_3+1}^{(r)} < q_3$, so $\nu_{p_3+1}^{(r)}+1 < q_3+1$.  Therefore, $(p_3+1,q_3+1) \notin {\calT}^{(r)}_{\munurplus}$.  
So $(p_3+1,q_3+1) \in {\calN}^{(r)}_{\munurplus}$, as desired.
\end{proof}

\begin{lem}{\label{lem:notin}}
Suppose that $h(r,\c_{r})$ is a tunnel hook in the partial \GBPR~diagram $D_{\mu/\nu^{(r-1)}}^{(r-1)}$ and consider the partial \GBPR~diagram
 $D_{\mu/\nu^{(r)}}^{(r)}$ resulting from applying Step 2 of Procedure~\ref{A:THC}.
    If 
    $(p,q)\in {\calB}^{(r)}_{\mu/\nu^{(r)}},$ then 
    either $(p,q) \in {\calB}^{(r-1)}_{\mu/\nu^{(r-1)}}$ or $(p-1,q-1)\in \h(r,\c_r).$ 
\end{lem}



\begin{proof}
Suppose $(p,q)\in {\calB}^{(r)}_{\mu/\nu^{(r)}}.$  Thus, from Definition
 \ref{Def:Bordcells}, we have
 $\nu_p^{(r)}+1\le q \le \nu_{p-1}^{(r)}+1.$  Assume $(p-1,q-1) \notin \h(r, \c_r)$, since otherwise we are done.  Since $q \le \nu_{p-1}^{(r)}+1,$ we have $q-1 \le \nu_{p-1}^{(r)}$.  Since $(p-1,q-1) \notin \h(r,\c_r)$ but $q-1 \le \nu_{p-1}^{(r)},$ we must have $q-1 \le \nu_{p-1}^{(r-1)}$.  So $q \le \nu_{p-1}^{(r-1)}+1$.  Also, $\nu_p^{(r-1)} \le \nu_p^{(r)}$, so $\nu_p^{(r-1)}+1 \le \nu_p^{(r)}+1 \le q$.  Therefore $\nu_p^{(r-1)} +1 \le q$.  Putting these together implies $$\nu_p^{(r-1)}+1 \le q \le \nu_{p-1}^{(r-1)}+1.$$  Therefore, by definition, $(p,q) \in \calB_{\mu/\nu^{(r-1)}}^{(r-1)}$.
 \end{proof}

Lemma~\ref{T:Losecells} implies an algorithm for identifying the cells in
$\mathcal{T}^{(r-1)}_{\mu/\nu^{(r-1)}}$ available to become terminal cells at each step $r$ in the construction of a \THF.  Along the way, we uncover a permutation associated with each \THF.  This algorithm to identify \allowcells~and produce the associated permutation is described below.  The subscripts on the $\mathbb{T}$ emphasize the fact that each new collection of cells depends on the choice of cells in the previous step.

\begin{proc}{\label{Pr:perm}}
Let $\mu \in \mathbb{Z}^k$ and set $\nu=\nu^{(0)} = \emptyset$.  The following algorithm constructs a sequence of cells (which we will see are terminal cells for a \allowhookfilling) and also produces a permutation associated to each choice of cells.

\begin{enumerate}
\item Let $\mathbb{T}_{\mu/\nu^{(0)}} = \{(1,1),(2,1),\ldots, (k,1)\}$ be the collection of all cells in the leftmost column of the \GBPR~diagram $D_{\mu}$.

\item Select a cell $\c_1=(p_1,1)$ from $\mathbb{T}_{\mu/\nu^{(0)}}$ and set $$\mathbb{T}_{\mu/\nu^{(1)}}=\{(2,2),(3,2),\ldots, (p_1,2),(p_1+1,1),\ldots,(k,1)\};$$ $\mathbb{T}_{\mu/\nu^{(1)}}$ is the set constructed from $\mathbb{T}_{\mu/\nu^{(0)}}$ by removing $(p_1,1)$ and adding $(1,1)$ to each cell in $\mathbb{T}_{\mu/\nu^{(0)}}$ situated in a row lower than row $p_1$.

\item Let $\sigma_1=p_1-1+1=p_1$.  Note that $\mathcal{L}_{\sigma_1}$ (see Equation~\eqref{E:LA}) is the diagonal containing $(p_1,1)$.

\item Repeat the following steps, once for each value of $r$ from $2$ to $k$.

\begin{enumerate}
\item Select a cell $\c_r=(p_r,q_r)$ from $\mathbb{T}_{\mu/\nu^{(r-1)}}$.

\item{\label{add11}} Construct $\mathbb{T}_{\mu/\nu^{(r)}}$ from $\mathbb{T}_{\mu/\nu^{(r-1)}}$ by removing $\c_r$ and adding $(1,1)$ to each of cell in $\mathbb{T}_{\mu/\nu^{(r-1)}}$ from a row lower than row $p_r$.

\item Let $\sigma_r=p_r-q_r+1$.  Note that $\mathcal{L}_{\sigma_r}$ is the diagonal containing $(p_r,q_r)$.
\end{enumerate}

\item Set $\sigma=(\sigma_1, \sigma_2, \hdots , \sigma_k)$.
\end{enumerate}
\end{proc}

\begin{lem}{\label{L:diags}}
The cells in the set $\mathbb{T}_{\mu/\nu^{(r-1)}}$ are precisely the \allowcells~in $\mathcal{T}^{(r-1)}_{\mu/\nu^{(r-1)}}$ for $1 \le r \le k$.  Furthermore, the resulting sequence $\sigma=(\sigma_1, \hdots , \sigma_k)$ is a permutation in $S_k$.
\end{lem}

\begin{proof}
 Let $\mathbb{T}_{\mu/\nu^{(0)}}=\{(1,1),(2,1),\ldots, (k,1)\}$ be the collection of cells in the leftmost column of $D^{(0)}_{\mu/\nu^{(0)}}$.  At the beginning of Step $r=1$ of the procedure, it is clear that $\mathbb{T}_{\mu/\nu^{(0)}}= \mathcal{T}^{(0)}_{\mu/\nu^{(0)}}$.  Note that each cell of $\mathcal{T}^{(0)}_{\mu/\nu^{(0)}}$
belongs to a unique diagonal $\mathcal{L}_j$ for $1 \le j \le k.$

Let $\tau_1=(p,1) \in \mathcal{L}_p$ be the cell chosen as terminal cell in Step 1 of Procedure~\ref{Pr:perm}.  Then Lemma~\ref{T:Losecells} (A) implies $(p+1,2)\in \mathcal{N}^{(1)}_{\mu/\nu^{(1)}}.$  Furthermore, by repeated application of Lemma~\ref{lem:notin} (D) and (E), no elements of $\mathcal{L}_p$ other than $(p,1)$ can be in any of the collections $\mathcal{T}^{(r)}_{\mu/\nu^{(r)}}$ for any $1 \le r \le k-1.$ 

Consider an arbitrary cell $(s,1)$ with $s \not= p$.  By Lemma~\ref{T:Losecells} (C), if  $\h(1,\tau_1)$ covers
the cell $(s,1)\in \mathcal{T}^{(0)}_{\mu/\nu^{(0)}}$, then $(s+1,2) \in \mathcal{T}^{(1)}_{\mu/\nu^{(1)}}$.  If $\h(1,\tau_1)$ does not cover
the cell $(s,1)$, then Lemma~\ref{T:Losecells} (B) implies
$ (s,1)\in\mathcal{T}^{(1)}_{\mu/\nu^{(1)}}.$
This proves the containment $$\mathbb{T}_{\mu / \nu^{(1)}}=\{(2,2),(3,2),\ldots, (p,2),(p+1,1),\ldots,(k,1)\} \subseteq \mathcal{T}^{(1)}_{\mu/\nu^{(1)}}.$$  Since there are exactly $k-1$ cells in $\mathcal{T}^{(1)}_{\mu / \nu^{(1)}}$, the fact that there are $k-1$ cells in $\mathbb{T}_{\mu / \nu^{(1)}}$ implies that $\mathbb{T}_{\mu / \nu^{(1)}}=\mathcal{T}^{(1)}_{\mu/\nu^{(1)}}.$

This argument can be repeated to show that each collection $\mathbb{T}_{\mu/\nu^{(r)}}$ is equal to the set $\mathcal{T}^{(r)}_{\mu/\nu^{(r)}}$ and includes at most one cell from $\mathcal{L}_j$ for each $1 \le j \le k$, since adding $(1,1)$ to a cell does not change the diagonal in which it lies.  

The construction of the cells in $\mathbb{T}_{\mu/\nu^{(r)}}$ from those in $\mathbb{T}_{\mu/\nu^{(r-1)}}$ removes a cell in row $r$ or higher and increases the row value for each cell lower than the removed cell.   Therefore by induction, the cells in $\mathbb{T}_{\mu/\nu^{(r)}}$ all lie in rows strictly higher than row $r$.

Finally, repeated application of Lemma~\ref{T:Losecells} (A, D, E) implies that if $\mathcal{L}_{\sigma_r}$ is the diagonal containing the cell $\c_r$ removed during Step $r$, then no cell from diagonal $\mathcal{L}_{\sigma_r}$ can appear in $\mathcal{T}^{(i)}_{\mu/\nu^{(i)}}$ for $i>r$.  Therefore each of the $k$ diagonals $\mathcal{L}_{\sigma_1}, \mathcal{L}_{\sigma_2}, \hdots , \mathcal{L}_{\sigma_k}$ removed during the $k$ steps of the procedure is distinct, and satisfy $1 \le \sigma_j \le k$ for all $j$.  This implies that $\sigma$ is indeed a permutation in $S_k$.
\end{proof}

Lemma~\ref{L:diags} in fact proves that Procedure~\ref{Pr:perm} is equivalent to the \THF~construction procedure, since selecting a \allowhook~starting in row $r$ can be done by simply selecting its terminal cell.  Every possible terminal cell for a \allowhook~starting in row $r$ (for the diagram $D^{(r-1)}_{\mu/\nu^{(r-1)}}$ created from the selection of the first $r-1$ \allowhooks) is included in $\mathbb{T}_{\mu/\nu^{(r-1)}}$.  

The following example demonstrates this construction.  

\begin{ex}
Let $\mu$ be a permutation of length $k=10$ and consider step $i=6$ in the \allowhookfilling~construction.  Assume the first $5$ \allowhooks~have been constructed and $\sigma_i \in \{2,3,6,7,9\}$ for $1 \le i \le 5$.  Then the \allowcells~in $\mathbb{T}_{\mu/\nu^{(5)}}$ are in diagonals $\mathcal{L}_1$,
$\mathcal{L}_{4}$, $\mathcal{L}_{5}$, $\mathcal{L}_{8}$, and $\mathcal{L}_{10}$.  If a cell is in $\mathcal{L}_j$, any cell northwest of this cell will be in a higher diagonal, which means that the \allowcell~in row $6$ must be in $\mathcal{L}_1$, the \allowcell~in row $7$ must be in $\mathcal{L}_4$, etc.  Therefore, $$\mathbb{T}_{\mu/\nu^{(5)}} = \{(6,6),(7,4),(8,4),(9,2),(10,1)\}.$$  If, for example, the cell $(8,4)$ is selected as the terminal cell for the next \allowhook, then the new collection of \allowcells~will become 
$$\mathbb{T}_{\mu/\nu^{(6)}} = \{(7,7),(8,5),(9,2),(10,1)\},$$ since $(1,1)$ is added to the first two entries of $\mathbb{T}_{\mu/\nu^{(5)}}$ and $(8,4)$ is removed.

If, instead, the cell $(9,2)$ is selected from $\mathbb{T}_{\mu/\nu^{(5)}}$ for the next terminal cell, the new collection of \allowcells~will become $$\mathbb{T}_{\mu/\nu^{(6)}} = \{(7,7),(8,5),(9,5),(10,1)\},$$ since $(1,1)$ is added to the first three entries of $\mathbb{T}_{\mu/\nu^{(5)}}$ and $(9,2)$ is removed.
\end{ex}

The following proposition, which states that there is a bijection between
\allowhookfillings~and permutations, will be used in the proof of our combinatorial interpretation of the $\NSym$ Jacobi-Trudi determinant.

\begin{prop}
Let $\mu=(\mu_1, \hdots , \mu_k)$ be a sequence.   
There is a bijection between \allowhookfillings~of the \GBPR~diagram for $\mu$ and  permutations $\sigma\in S_k.$
\label{T:perm}
\end{prop}

\begin{proof}
We first show how to determine the permutation from the \allowhookfilling~of the \GBPR~diagram.  Let $$\gamma=\{\h(1,(p_1,q_1)),\h(2,(p_2,q_2)),\ldots, \h(k,(p_k,q_k))\}$$ be a \allowhookfilling~(\THC) of the \GBPR~diagram for $\mu$, where $\c_i=(p_i,q_i)$ for $1 \le i \le k$.  Let $\sigma(\gamma)$ be the permutation whose $i^{th}$ part is given by $p_i-q_i+1$.  This is precisely the permutation produced by Procedure~\ref{Pr:perm}. 
 Lemma~\ref{L:diags} guarantees that $\sigma$ is indeed a permutation in $S_k$.

To recover the \THF~from the permutation $\sigma=(\sigma_1, \sigma_2, \hdots , \sigma_k)$, it is enough to determine the terminal cells $\{\c_1, \c_2, \hdots , \c_k\}$.  To this end, set $\c_r = (\sigma_r+m,1+m)$, where 
$$m= \# \{ i \in \{1,\ \hdots, k \} | i<r \textrm{ and } \sigma_i > \sigma_r \}.$$  We show that these cells are indeed the terminal cells for a \THF~by proving inductively that $\c_r \in \mathbb{T}_{\mu/\nu^{(r-1)}}$ for $1 \le r \le k$.  

For the base case, note that $\c_1=(\sigma_1,1)$, since there is no value $i \in \{ 1 , \hdots , k \}$ such that $i<1$.  Since $(\sigma_1, 1) \in \mathbb{T}_{\mu/\nu^{(0)}}$, the base case is satisfied.

Next assume that $\c_j \in \mathbb{T}_{\mu/\nu^{(j-1)}}$ for $1 \le j < r$.  The cell in diagonal $\mathcal{L}_{\sigma_{r}}$ of $\mathbb{T}_{\mu/\nu^{(0)}}$ is $(\sigma_{r},1)$.  With each choice of $\c_j$ (for $j<r$) selected during Procedure~\ref{Pr:perm}, this cell either remains the same (if $\sigma_j < \sigma_r$) or is augmented by $(1,1)$ (if $\sigma_j > \sigma_r$).  But this implies that $\c_{r} \in \mathbb{T}_{\mu/\nu^{(r-1)}}$ since $m$ counts the number of times (in steps $1$ through $r-1$) the coordinates of this cell is increased by $(1,1)$.

Thus, the cells $\{\c_1, \c_2, \hdots , \c_k\}$ are indeed \allowcells~and produce precisely the \allowhookfilling~associated to $\sigma$ from Procedure~\ref{Pr:perm}. 
 Therefore, the map described is indeed a bijection between \allowhookfillings~of $D_{\mu}$ and permutations in $S_k$.
\end{proof}

\begin{ex}
\label{E:gettunnel}
Let $\sigma=(4,7,3,1,6,2,5) \in S_{7}.$  Then $\c_1=(4,1)$.  Since $m=0$ for $\sigma_2$, we have $\c_2=(\sigma_2+0,1+0)=(7,1)$.  For $\sigma_3$, we have $m=2$ since both $\sigma_1$ and $\sigma_2$ are greater than $\sigma_3=3$.  So $\c_3=(3+2,1+2)=(5,3)$.  Continuing this process produces the overall collection $$\{(4,1),(7,1),(5,3),(4,4),(7,2),(6,5),(7,3)\}$$ of terminal cells for the \allowhookfilling~corresponding to $\sigma$.  Note that these terminal cells are in diagonals $\mathcal{L}_4, \mathcal{L}_7, \mathcal{L}_3, \mathcal{L}_1, \mathcal{L}_6, \mathcal{L}_2, \mathcal{L}_5,$ respectively.    
\end{ex}

The bijection between \THF's and permutations allows us to go one step further in giving a combinatorial interpretation for the partitions \\$(\nu^{(r)}_{r+1}, \nu^{(r)}_{r+2}, \hdots , \nu^{(r)}_k)$ appearing in the partial \GBPR~diagrams $D^{(r)}_{\mu/\nu^{(r)}}$ associated to $\sigma=(\sigma_1,\sigma_2,\hdots,\sigma_k).$  In the following, assume $\sigma\in S_k$ and $m<k$.  Let $\beta_{\sigma,m}$ be the collection $\{\sigma_{m+1},\sigma_{m+2}, \cdots, \sigma_{k}\}$ and set $\beta_{\sigma,m,j}$ to be the $j^{th}$ smallest element of $\beta_{\sigma,m}$.

\begin{ex}
If $\sigma  =(4,3,7,9,2,6,5,8,1)$
then $\beta_{\sigma,4}=\{1,2,5,6,8\}$
and $\beta_{\sigma,4,3}=5.$
\end{ex}

Note that the set $\beta_{\sigma,m}$ is the collection of diagonals $\mathcal{L}_j$ remaining after the first $m$ \allowhooks~for the \THF~associated to $\sigma$ have been constructed.

\begin{lem}
Let 
$\sigma=(\sigma_1,\sigma_2,\ldots,\sigma_k)\in S_k$ and let \\
$\{D^{(0)}_{\mu/\nu^{(0)}},D^{(1)}_{\mu/\nu^{(1)}}, \hdots, D^{(k)}_{\mu/\nu^{(k)}} \}$
be the sequence of partial \GBPR~diagrams produced during the construction of the corresponding \allowhookfilling.  Then
$\nu_j^{(r)}$ (for $j > r$) equals the number of entries in
$\{\sigma_1,\sigma_2,\ldots,\sigma_r\}$ that are greater than $\beta_{\sigma,r,j-r}.$
\label{T:countnuj}
\end{lem}

\begin{proof}
The \allowhookfilling~construction algorithm begins with the collection of cells in the leftmost column, one on each diagonal $\mathcal{L}_a$ by Lemma~\ref{lem:notin}.  At each stage of the algorithm, one diagonal is removed and each \allowcell~in a diagonal below the removed diagonal is moved one unit north and one unit east.  This means that after $r$ iterations of the \allowhookfilling~construction, the \allowhook~beginning in row $r+1$ and ending in row $j>r$ terminates in the $(j-r)^{th}$ smallest remaining diagonal $\mathcal{L}_b$, where $b=\beta_{\sigma,r,j-r}$.  The coordinates of this \allowcell~are $(p_{r+1},q_{r+1})$, where $q_{r+1}$ is one more than the number of \allowhooks~terminating in a diagonal greater than $b$.  But since $(p_{r+1},q_{r+1})$ is a \allowcell, we have $\nu_j^{(r)}=q_{r+1}-1$, so $\nu_j^{(r)}$ equals the number of \allowhooks~terminating in a diagonal greater than $\beta_{\sigma,r,j-r}$.  Since $\sigma_i$ equals the diagonal in which the $i^{th}$ \allowhook~terminates, $\nu_j^{(r)}$ equals the number of entries in $\{ \sigma_1, \sigma_2, \hdots , \sigma_r\}$ greater than $\beta_{\sigma,r,j-r}$.
\end{proof}

Recall that the \emph{Lehmer code}~\cite{Leh60} $L(\sigma)$ of a permutation $\sigma \in S_k$ is given by $$L(\sigma) = (L(\sigma_1),L(\sigma_2), \hdots , L(\sigma_k)),$$ where $L(\sigma_i) = \# \{ j>i | \sigma_j < \sigma_i \}$. 
 Note, if $d:=\sum_{i=1}^{n-1}L(\sigma_i)$ counts the total number of inversions of $\sigma,$ then the sign of $\sigma$ is given by $\sign(\sigma)=(-1)^{d}.$

The following lemma provides a method for counting the number of rows covered by a \allowhook.  This will be useful in constructing the sign of a \allowhookfilling.

\begin{lem}
Let $$\h(1,\tau_1),\h(2,\tau_2),\ldots,
\h(k,\tau_k),
$$
be a \THF~of $D_{\mu}$ with corresponding permutation 
$$\sigma=(\sigma_1,\sigma_2,\ldots, \sigma_k)\in S_k.$$
The number of rows covered by $\h(r,\c_r)$
is equal to $L(\sigma_{r})+1$.
\label{T:signs}
\end{lem}

\begin{proof}
The \allowhook~$\h(r,(p_r,q_r))$ begins in row $r$ and travels through each \allowcell~situated in a diagonal smaller than $\sigma_r$, which means the number of rows covered by $\h(r,(p_r,q_r))$ equals one plus the number of remaining diagonals in $\mathbb{T}_{\mu/\nu^{(r-1)}}$ smaller than $\sigma_r$.  But this is equal to one plus the number of entries in $\{ \sigma_{r+1}, \sigma_{r+2} , \hdots , \sigma_k\}$ which are less than $\sigma_r$.  But this is precisely equal to $L(\sigma_{r})+1$, as desired.
\end{proof}

Notice that the terminal cells for the \allowhooks~are not dependent on $\mu$.  However, the \allowhooks~themselves still vary based on $\mu$ due to the cells in the lowest row of each \allowhook~as well as the new red cells that are potentially introduced, depending on $\mu$.

\begin{ex}
Let $\sigma=(4,7,3,1,6,2,5) \in S_{7}.$  Then $L(\sigma)=(3,5,2,0,2,0,0)$. Recall from Example~\ref{E:gettunnel} that the terminal cells in the \THC~corresponding to $\sigma$ are $$\{(4,1),(7,1),(5,3),(4,4),(7,2),(6,5),(7,3)\}.$$  Consider, for example, the \allowhook~$\h(3,(5,3))$.  This \allowhook~starts in row $3$ and ends in row $5$, so it covers $3$ rows, which equals $L(\sigma_3)+1$.
\end{ex}

Our next step in the proof of Theorem \ref{T:main} is to show that the lengths of the \allowhooks~equal the subscripts of the 
corresponding entries in $M_\mu.$

\begin{lem} 
Let $\mu$ be a sequence of length $k$ and let $M_{\mu}$ be defined as above.  Assume the first $r-1$ \allowhooks~have been constructed so that their terminal cells do not lie on $\mathcal{L}_j$.  Then
\begin{equation}
  (M_{\mu})_{r,j}(=H_{\mu_r-r+j})=H_{\Delta(\h(r,\c))},
\end{equation}
where $\c$ is the unique cell in the diagonal $\mathcal{L}_j$ that is also in $\mathbb{T}_{\murmin/\nu^{(r-1)}}$.
\label{T:Mrj}
\end{lem}

\begin{proof}
First note that there exists a unique $s$ such that $(j+s,1+s)=\c \in \mathbb{T}_{\mu/\nu^{(r-1)}}$ by the arguments given in the proof of Lemma~\ref{L:diags}.  
By the definition of $taxi$ (Equation \ref{Def:taxi}), we have $$taxi(\h(r,\c))=(\nu_r^{(r-1)}+1-(1+s))+(j+s-r)=\nu_r^{(r-1)}-r+j.$$

Next, recognize that $\mu_r - \nu_r^{(r-1)} = b_r-c_r= \spin(r)$, where the first equality is due to Lemma~\ref{lem:gbprsum} (since $\nu_r^{(r-1)}=a_r$) and the second is the definition of $\spin$ (see Definition~\ref{D:bank}).  Therefore, $$\spin(r)+taxi(\h(r,\c)) = \mu_r-\nu_r^{(r-1)} + \nu_r^{(r-1)} - r + j=\mu_r-r+j,$$ which is exactly what is necessary since $\Delta(\h(r,\c))=\spin(r)+taxi(\h(r,\c))$.
\end{proof}

\subsection{Proof of Theorem~\ref{T:main}}

We now have all the pieces we need to complete the proof of Theorem~\ref{T:main}.

\begin{proof}[Proof of Theorem~\ref{T:main}]
Recall~\cite{BBSSZ14} that
$$\I_\mu=\ndet(M_\mu)=\sum_{\sigma\in S_k}
\sign(\sigma) (M_\mu)_{i,\sigma_i}.
$$
 Proposition~\ref{T:perm} gives a bijection between permutations and tunnel hook coverings.
In Lemma~\ref{T:Mrj} we show that the subscripts on the entries of the matrix $M_{\mu}$ are equal to the $\Delta$ values for the corresponding \allowhooks.
All that remains is to show that
$\sign(\sigma)$ equals the product of the signs of the \allowhooks.
Recall that the sign of a permutation $\sigma=(\sigma_1,\sigma_2, \cdots, \sigma_k)$ is given by
$(-1)^d,$ where $d=\sum_i L(\sigma_i)$.  

Lemma~\ref{T:signs}  shows that the sign of each individual \allowhook~$h(r,\tau_i)$
is in fact $(-1)^{L(\sigma)_i},$ since the sign of an individual hook $\h(i,\c)$ is one less than the number of rows covered by $\h(i,\c)$.
Thus, the product of the signs of the \allowhooks~gives the sign of the permutation corresponding to that \allowhookfilling.
\end{proof}

In Examples \ref{Ex:CompSHF313},
\ref{Ex:CompSHF303}, and \ref{Ex:CompSHF3n13},
we give the complete collection
of \allowhookfillings~for shapes $(3,1,3),(3,0,3),$ and $(3,-1,3)$ to provide the decompositions of $\I_{(3,1,3)}$, $\I_{(3,0,3)},$ and $\I_{(3,-1,3)}$ into complete homogeneous noncommutative symmetric functions.

\begin{ex}
Below is a complete decomposition of $\I_{(3,1,3)}.$

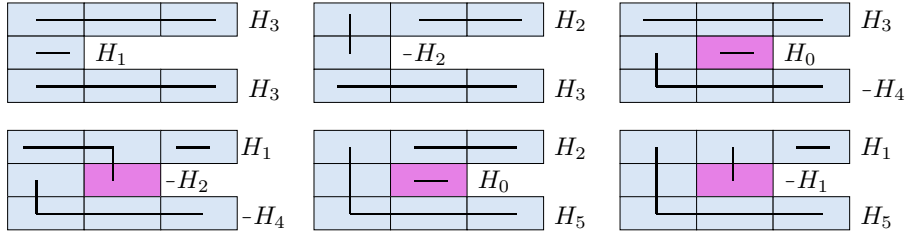
\begin{picture}(385,100)
\scalebox{0.965}{
\put(-15,50){
\vbox{
\Yfillcolour{lightblue}
\young(\ \ \  )
\vskip -1.295pt
\young(\ )
\vskip -1.295pt
\young(\ \ \  )
}}

\put(12,53){\line(1,0){70}}
\put(12,66){\line(1,0){13}}
\put(12,79){\line(1,0){70}}

\put(95,49){$H_3$}
\put(35,63){$H_1$}
\put(95,76){$H_3$}

\put(105,50){
\vbox{
\Yfillcolour{lightblue}
\young(\ \ \  )
\vskip -1.295pt
\young(\ )
\vskip -1.295pt
\young(\ \ \  )
}}

\put(130,53){\line(1,0){70}}
\put(135,66){\line(0,1){15}}
\put(162,79){\line(1,0){40}}

\put(215,49){$H_3$}
\put(155,63){$\shortminus H_2$}
\put(215,76){$H_2$}

\put(225,50){
\vbox{
\Yfillcolour{lightblue}
\young(\ \ \  )
\vskip -1.295pt
\young(\ !<\Yfillcolour{lightpurple}>\  )
\vskip -1.295pt
\young(\ \ \  )
}}

\put(255,53){\line(1,0){65}}
\put(255,53){\line(0,1){13}}
\put(280,66){\line(1,0){13}}
\put(250,79){\line(1,0){70}}

\put(335,49){$\shortminus H_4$}
\put(305,63){$H_0$}
\put(335,76){$H_3$}

\put(-15,0){
\vbox{
\Yfillcolour{lightblue}
\young(\ \ \  )
\vskip -1.295pt
\young(\ !<\Yfillcolour{lightpurple}>\  )
\vskip -1.295pt
\young(\ \ \  )
}}

\put(12,3){\line(1,0){65}}
\put(12,3){\line(0,1){13}}
\put(42,16){\line(0,1){13}}
\put(7,29){\line(1,0){35}}
\put(67,29){\line(1,0){13}}

\put(92,-1){$\shortminus H_4$}
\put(62,13){$\shortminus H_2$}
\put(92,26){$H_1$}

\put(105,0){
\vbox{
\Yfillcolour{lightblue}
\young(\ \ \  )
\vskip -1.295pt
\young(\ !<\Yfillcolour{lightpurple}>\  )
\vskip -1.295pt
\young(\ \ \  )
}}

\put(135,3){\line(1,0){65}}
\put(135,3){\line(0,1){26}}
\put(160,16){\line(1,0){13}}
\put(160,29){\line(1,0){40}}

\put(215,-1){$H_5$}
\put(185,13){$H_0$}
\put(215,26){$H_2$}

\put(225,0){
\vbox{
\Yfillcolour{lightblue}
\young(\ \ \  )
\vskip -1.295pt
\young(\ !<\Yfillcolour{lightpurple}>\  )
\vskip -1.295pt
\young(\ \ \  )
}}

\put(255,3){\line(1,0){65}}
\put(255,3){\line(0,1){26}}
\put(285,16){\line(0,1){13}}
\put(310,29){\line(1,0){13}}

\put(335,-1){$H_5$}
\put(305,13){$\shortminus H_1$}
\put(335,26){$H_1$}}
\end{picture}

\vskip .125truein\noindent
Putting this together produces the decomposition $$\I_{(3,1,3)}=
H_{(3,1,3)}-H_{(3,2,2)}-H_{(4,3)}+H_{(4,2,1)}+H_{(5,2)}-H_{(5,1,1)}.$$

\label{Ex:CompSHF313}
\end{ex}

Example~\ref{Ex:CompSHF313} does not include red cells in any of the \allowhookfillings, since no parts are negative and it is not possible in this example for a \allowhook~to cover a purple cell at any point other than in its originating row.  Note that it is, however, possible to introduce a red cell in the decomposition of an immaculate function whose indexing composition has all positive parts; consider for example the \allowhook~decomposition of $\I_{1,2,1}$ (not shown), which includes the term $H_3H_2H_{-1}$. 

The next example illustrates a situation in which it is possible to introduce a red cell, despite the fact that the indexing composition does not contain any negative parts.  

\begin{ex}
A complete decomposition of $\I_{(3,0,3)}.$
\vskip -0.15truein
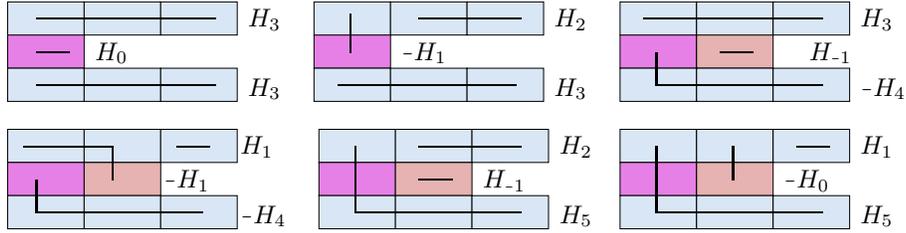
\begin{picture}(385,100)
\scalebox{0.965}{
\put(-15,50){
\vbox{
\Yfillcolour{lightblue}
\young(\ \ \  )
\vskip -1.295pt
\Yfillcolour{lightpurple}
\young(\ )
\vskip -1.295pt
\Yfillcolour{lightblue}
\young(\ \ \  )
}}

\put(12,53){\line(1,0){70}}
\put(12,66){\line(1,0){13}}
\put(12,79){\line(1,0){70}}

\put(95,49){$H_3$}
\put(35,63){$H_0$}
\put(95,76){$H_3$}

\put(105,50){
\vbox{
\Yfillcolour{lightblue}
\young(\ \ \  )
\vskip -1.295pt
\Yfillcolour{lightpurple}
\young(\ )
\vskip -1.295pt
\Yfillcolour{lightblue}
\young(\ \ \  )
}}

\put(130,53){\line(1,0){70}}
\put(135,66){\line(0,1){15}}
\put(162,79){\line(1,0){40}}

\put(215,49){$H_3$}
\put(155,63){$\shortminus H_1$}
\put(215,76){$H_2$}

\put(225,50){
\vbox{
\Yfillcolour{lightblue}
\young(\ \ \  )
\vskip -1.295pt
\Yfillcolour{lightpurple}
\young(\ !<\Yfillcolour{lightred}>\  )
\vskip -1.295pt
\Yfillcolour{lightblue}
\young(\ \ \  )
}}

\put(255,53){\line(1,0){65}}
\put(255,53){\line(0,1){13}}
\put(280,66){\line(1,0){13}}
\put(250,79){\line(1,0){70}}

\put(335,49){$\shortminus H_4$}
\put(315,63){$H_{\shortminus 1}$}
\put(335,76){$H_3$}

\put(-15,0){
\vbox{
\Yfillcolour{lightblue}
\young(\ \ \  )
\vskip -1.295pt
\Yfillcolour{lightpurple}
\young(\ !<\Yfillcolour{lightred}>\  )
\vskip -1.295pt
\Yfillcolour{lightblue}
\young(\ \ \  )
}}

\put(12,3){\line(1,0){65}}
\put(12,3){\line(0,1){13}}
\put(42,16){\line(0,1){13}}
\put(7,29){\line(1,0){35}}
\put(67,29){\line(1,0){13}}

\put(92,-1){$\shortminus H_4$}
\put(62,13){$\shortminus H_1$}
\put(92,26){$H_1$}

\put(107,0){
\vbox{
\Yfillcolour{lightblue}
\young(\ \ \  )
\vskip -1.295pt
\Yfillcolour{lightpurple}
\young(\ !<\Yfillcolour{lightred}>\  )
\vskip -1.295pt
\Yfillcolour{lightblue}
\young(\ \ \  )
}}

\put(137,3){\line(1,0){65}}
\put(137,3){\line(0,1){26}}
\put(162,16){\line(1,0){13}}
\put(162,29){\line(1,0){40}}

\put(217,-1){$H_5$}
\put(187,13){$H_{\shortminus 1}$}
\put(217,26){$H_2$}

\put(225,0){
\vbox{
\Yfillcolour{lightblue}
\young(\ \ \  )
\vskip -1.295pt
\Yfillcolour{lightpurple}
\young(\ !<\Yfillcolour{lightred}>\  )
\vskip -1.295pt
\Yfillcolour{lightblue}
\young(\ \ \  )
}}

\put(255,3){\line(1,0){65}}
\put(255,3){\line(0,1){26}}
\put(285,16){\line(0,1){13}}
\put(310,29){\line(1,0){13}}

\put(335,-1){$H_5$}
\put(305,13){$\shortminus H_0$}
\put(335,26){$H_1$}}
\end{picture}

\vskip .1truein \noindent
Including signs and recalling that $H_j=0$ if $j \in \mathbb{Z}^-$,
we see that
$$\I_{(3,0,3)}=
H_{(3,3)}-H_{(3,1,2)}+H_{(4,1,1)}-H_{(5,1)}.$$

\label{Ex:CompSHF303}
\end{ex}

Finally, Example~\ref{Ex:CompSHF3n13} depicts the complete decomposition of $\I_{(3,-1,3)}$ into \allowhooks.  In this situation, even though we begin with a red cell, two of the six \allowhookfillings~result in nonnegative indices.

\begin{ex}
A complete decomposition of $\I_{(3,-1,3)}:$

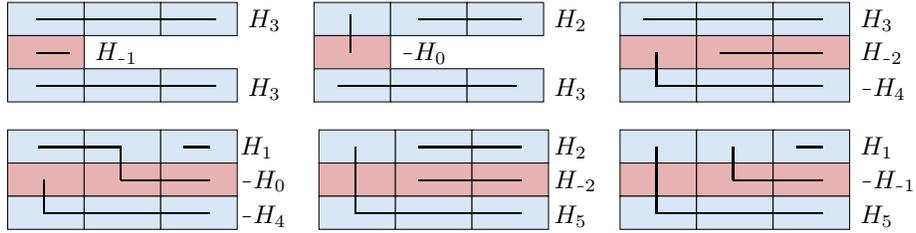
\begin{picture}(385,100)
\scalebox{0.965}{
\put(-15,50){
\vbox{
\Yfillcolour{lightblue}
\young(\ \ \  )
\vskip -1.295pt
\Yfillcolour{lightred}
\young(\ )
\vskip -1.295pt
\Yfillcolour{lightblue}
\young(\ \ \  )
}}

\put(12,53){\line(1,0){70}}
\put(12,66){\line(1,0){13}}
\put(12,79){\line(1,0){70}}

\put(95,49){$H_3$}
\put(35,63){$H_{\shortminus 1}$}
\put(95,76){$H_3$}

\put(105,50){
\vbox{
\Yfillcolour{lightblue}
\young(\ \ \  )
\vskip -1.295pt
\Yfillcolour{lightred}
\young(\ )
\vskip -1.295pt
\Yfillcolour{lightblue}
\young(\ \ \  )
}}

\put(130,53){\line(1,0){70}}
\put(135,66){\line(0,1){15}}
\put(162,79){\line(1,0){40}}

\put(215,49){$H_3$}
\put(155,63){$\shortminus H_0$}
\put(215,76){$H_2$}

\put(225,50){
\vbox{
\Yfillcolour{lightblue}
\young(\ \ \  )
\vskip -1.295pt
\Yfillcolour{lightred}
\young(\ \ \  )
\vskip -1.295pt
\Yfillcolour{lightblue}
\young(\ \ \  )
}}

\put(255,53){\line(1,0){65}}
\put(255,53){\line(0,1){13}}
\put(280,66){\line(1,0){40}}
\put(250,79){\line(1,0){70}}

\put(335,49){$\shortminus H_4$}
\put(335,63){$H_{\shortminus 2}$}
\put(335,76){$H_3$}

\put(-15,0){
\vbox{
\Yfillcolour{lightblue}
\young(\ \ \  )
\vskip -1.295pt
\Yfillcolour{lightred}
\young(\ \ \  )
\vskip -1.295pt
\Yfillcolour{lightblue}
\young(\ \ \  )
}}

\put(15,3){\line(1,0){65}}
\put(15,3){\line(0,1){13}}
\put(45,16){\line(0,1){13}}
\put(13,29){\line(1,0){32}}
\put(45,16){\line(1,0){35}}
\put(70,29){\line(1,0){10}}

\put(92,-1){$\shortminus H_4$}
\put(92,13){$\shortminus H_0$}
\put(92,26){$H_1$}

\put(107,0){
\vbox{
\Yfillcolour{lightblue}
\young(\ \ \  )
\vskip -1.295pt
\Yfillcolour{lightred}
\young(\ \ \ )
\vskip -1.295pt
\Yfillcolour{lightblue}
\young(\ \ \  )
}}

\put(137,3){\line(1,0){65}}
\put(137,3){\line(0,1){26}}
\put(162,16){\line(1,0){40}}
\put(162,29){\line(1,0){40}}

\put(215,-1){$H_5$}
\put(215,13){$ H_{\shortminus 2}$}
\put(215,26){$H_2$}

\put(225,0){
\vbox{
\Yfillcolour{lightblue}
\young(\ \ \  )
\vskip -1.295pt
\Yfillcolour{lightred}
\young(\ \ \  )
\vskip -1.295pt
\Yfillcolour{lightblue}
\young(\ \ \  )
}}

\put(255,3){\line(1,0){65}}
\put(255,3){\line(0,1){26}}
\put(285,16){\line(0,1){13}}
\put(285,16){\line(1,0){35}}
\put(310,29){\line(1,0){10}}

\put(335,-1){$H_5$}
\put(335,13){$\shortminus H_{\shortminus 1}$}
\put(335,26){$H_1$}}
\end{picture}

\vskip .125truein\noindent
Including signs and recalling that $H_j=0$ if $j \in \mathbb{Z}^-$,
we see that
$$\I_{(3,-1,3)}=
-H_{(3,2)} +H_{(4,1)}.$$

\label{Ex:CompSHF3n13}

\end{ex}

\subsection{Quasisymmetric Functions}

The Hopf algebra $\QSym$ of quasisymmetric functions is dual to the algebra $\NSym$, satisfying the pairing $$\langle \cdot, \cdot \rangle : \NSym \times \QSym \rightarrow \mathbb{Q},$$ defined by setting $$\langle H_{\alpha}, M_{\beta} \rangle = \delta_{\alpha,\beta}.$$  Let $\mathcal{C}$ be the set of all compositions.  Any pair of bases $\{X_{\alpha}\}_{\alpha \in \mathcal{C}}$ in $\NSym$ and $\{Y_{\beta}\}_{\beta \in \mathcal{C}}$ in $\QSym$ satisfying $\langle X_{\alpha} , Y_{\beta} \rangle = \delta_{\alpha, \beta}$ are said to be \emph{dual} to one another.  Hence, the complete homogeneous basis for $\NSym$ is dual to the monomial basis for $\QSym$.

The \emph{ribbon Schur basis} $\{R_{\alpha}\}_{\alpha \in \mathcal{C}}$ for $\NSym$ can be defined in terms of the complete homogeneous basis by $$R_{\alpha} = \sum_{\alpha \preceq \beta}(-1)^{\ell(\alpha)-\ell(\beta)} H_{\beta},$$ where $\preceq$ is the refinement ordering on compositions.  The ribbon Schur basis for $\NSym$ is dual to the fundamental basis for $\QSym$~\cite{GKLLRT95,Ges84}. 

The basis in $\QSym$ dual to the immaculate basis is called the \emph{dual immaculate quasisymmetric function basis}~\cite{BBSSZ14}.  Elements of this basis are denoted by $\dI_{\alpha}$ and expand positively in the monomial, fundamental, and Young quasisymmetric Schur bases~\cite{LMvW13,AHM18}.  The following combinatorial formula for the expansion of the monomial basis for $\QSym$ into the dual immaculate basis is an immediate corollary of Theorem~\ref{T:main}, due to duality.  Intuitively, to expand $M_{\alpha}$ into dual immaculates, the shapes being covered may vary but the values associated to the tunnel hook coverings must be equal to the composition $\alpha$.

\begin{cor}
The decomposition of the monomial quasisymmetric functions into the dual immaculate quasisymmetric functions is given by the following formula.

\begin{equation}
M_{\alpha}=\sum_{\mu \models | \alpha|} \sum_{\substack{\gamma \in \SHF_{\mu}, \\ \Delta(\gamma)=\alpha}} \prod_{r=1
}^k \ \sign(\h(r,\c_r))\ \dI_{\Delta(\h(r,\c_r))},
\end{equation}
where $\mu$ is a composition of $| \alpha|$, $\SHF_{\mu}$ denotes the collection of \allowhookfillings~of a diagram of shape $\mu,$ $\h(r,\c_r)\in \gamma$, and the sign $\sign(\h(r,\c_r))$ and integer value $\Delta(\h(r,\c_r))$ associated to each \allowhook~$\h(r,\c_r)$ are as described above.
\end{cor}

\begin{ex}
Let $\alpha=(2,1,2)$. 
 The following are all \THF~$\gamma$ such that $\Delta(\gamma)=(2,1,2)$.


\hskip -.25truein
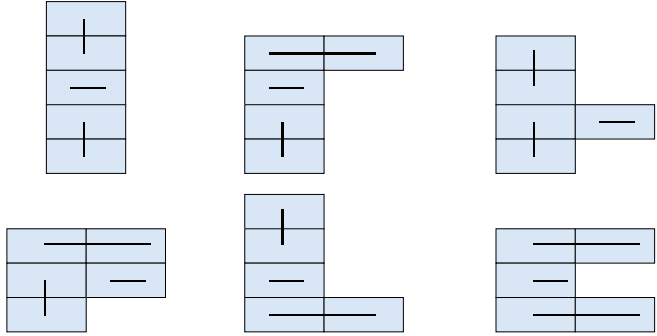
\begin{picture}(380,140)
\put(65,70){
\vbox{
\Yfillcolour{lightblue}
\young(\  )
\vskip -1.295pt
\Yfillcolour{lightblue}
\young(\ )
\vskip -1.295pt
\Yfillcolour{lightblue}
\young(\  )
\vskip -1.295pt
\Yfillcolour{lightblue}
\young(\  )
\vskip -1.295pt
\Yfillcolour{lightblue}
\young(\  )
}}

\put(95,73){\line(0,1){13}}
\put(90,99){\line(1,0){13}}
\put(95,112){\line(0,1){13}}

\put(140,70){
\vbox{
\Yfillcolour{lightblue}
\young(\ \  )
\vskip -1.295pt
\Yfillcolour{lightblue}
\young(\ )
\vskip -1.295pt
\Yfillcolour{lightblue}
\young(\  )
\vskip -1.295pt
\Yfillcolour{lightblue}
\young(\  )
}}

\put(170,73){\line(0,1){13}}
\put(165,99){\line(1,0){13}}
\put(165,112){\line(1,0){40}}

\put(235,70){
\vbox{
\Yfillcolour{lightblue}
\young(\  )
\vskip -1.295pt
\Yfillcolour{lightblue}
\young(\ )
\vskip -1.295pt
\Yfillcolour{lightblue}
\young(\ \  )
\vskip -1.295pt
\Yfillcolour{lightblue}
\young(\  )
}}

\put(265,73){\line(0,1){13}}
\put(290,86){\line(1,0){13}}
\put(265,100){\line(0,1){13}}

\put(50,10){
\vbox{
\Yfillcolour{lightblue}
\young(\ \ )
\vskip -1.295pt
\Yfillcolour{lightblue}
\young(\ \  )
\vskip -1.295pt
\Yfillcolour{lightblue}
\young(\  )
}}
\put(80,13){\line(0,1){13}}
\put(105,26){\line(1,0){13}}
\put(80,40){\line(1,0){40}}

\put(140,10){
\vbox{
\Yfillcolour{lightblue}
\young(\  )
\vskip -1.295pt
\Yfillcolour{lightblue}
\young(\  )
\vskip -1.295pt
\Yfillcolour{lightblue}
\young(\   )
\vskip -1.295pt
\Yfillcolour{lightblue}
\young(\ \  )
}}

\put(165,13){\line(1,0){40}}
\put(165,26){\line(1,0){13}}
\put(170,40){\line(0,1){13}}

\put(235,10){
\vbox{
\Yfillcolour{lightblue}
\young(\ \ )
\vskip -1.295pt
\Yfillcolour{lightblue}
\young(\  )
\vskip -1.295pt
\Yfillcolour{lightblue}
\young(\ \  )
}}

\put(265,13){\line(1,0){40}}
\put(265,26){\line(1,0){13}}
\put(265,40){\line(1,0){40}}
\end{picture}
Therefore, $M_{212} = \dI_{11111}-\dI_{1112}+\dI_{1211}-\dI_{122}-\dI_{2111}+\dI_{212}$.
\end{ex}

\section{Immaculate functions indexed by skew shapes}
\label{S:modifynu}

We now extend Theorem~\ref{T:NsymJacobiTrudi} to introduce a definition of skew immaculate functions, just as the Jacobi-Trudi formula can be used to define skew Schur functions in terms of the complete homogeneous symmetric functions. 

\begin{defn}{\label{Def:Imunu}} Let $\mu, \nu \in \mathbb{Z}^k$ be sequences of integers.  Recalling Theorem \ref{T:SymJacobiTrudi}, we define $(M_{\mu/\nu})_{i,j}=H_{(\mu_i-i)-(\nu_j-j)}$ and\begin{equation}\I_{\mu/\nu}=\ndet(M_{\mu/\nu}).\end{equation}
where the determinant $\ndet$ is expanded using Laplace expansion
starting in the top row and continuing sequentially to the bottom row.
\end{defn}

For example, if $\mu=(2,5,3)$ and $\nu=(1,3,0)$, then $$M_{\mu/\nu} = \begin{bmatrix} H_1 & H_0 & H_4  \\ H_{3} & H_{2} & H_6  \\ H_{0} & H_{-1} & H_3  \end{bmatrix}.$$  Recalling that $H_a=0$ if $a \in \mathbb{Z}^{-}$ and $H_0=1$, the resulting decomposition of $\I_{\mu/\nu}$ into the complete homogeneous basis for $\NSym$ is therefore $$\I_{\mu/\nu}=H_{(1,2,3)}-H_{(3,3)}+H_{(6)} - H_{(4,2)}.$$

 Notice that the skew Schur function $s_{(2,5,3)/(1,3)}$ becomes $s_{(4,3,3)/(2,2)}$ under the well-known \emph{straightening algorithm} which states that for any integer sequences $\lambda', \lambda''$ and any integers $a$ and $b$, $$s_{(\lambda',a,b,\lambda'')}=-s_{(\lambda',b-1,a+1,\lambda'')}.$$  This property arises by swapping rows in the Jacobi-Trudi determinant for Schur functions.  Note that this relationship does not generally apply to immaculate functions due to noncommutativity.  For example, $\I_{(a,b)}=H_{(a,b)}-H_{(a+1,b-1)}$ while $\I_{(b-1,a+1)}=H_{(b-1,a+1)}-H_{(b,a)}$ so that $\I_{(a,b)} \not= - \I_{(b-1,a+1)}$.
 
 Decomposing the Schur function from our example into the complete homogeneous symmetric functions produces the expansion $$s_{(4,3,3)/(2,2)} = h_{(3,2,1)}-h_{(3,3)}+h_{(6)}-h_{(4,2)},$$ which is exactly the decomposition obtained by applying the forgetful map to the expansion produced by our construction in $\NSym$.  This is true in general since the starting matrices are identical, the only difference between the $H-$expansion of $\I_{\mu/\nu}$ and the $h-$expansion of $s_{\mu/\nu}$ being that the indices for the $H-$basis expansions do not commute whereas the indices for the $h-$basis expansions do commute.  
 
It is important to note that the Hopf algebra operation $\rightharpoonup$ for skewing elements of $\NSym$ by elements of $\QSym$ (often referred to as the transpose of multiplication~\cite{Mon93}), another natural candidate for skew immaculates, does not always coincide with this construction.  For example, skewing $\I_{(2,5,3)}$ by $\dI_{(1,3)}$ produces the $H-$decomposition $$\dI_{(1,3)} \rightharpoonup \I_{(2,5,3)} =\I_{(1,2,3)}+\I_{(1,3,2)}+\I_{(1,4,1)}+\I_{(1,5)}+\I_{(2,4)} \not= \I_{(2,5,3)/(1,3)}.$$  One open problem is to classify the pairs of compositions for which the two skew candidates coincide.

 
 Theorem~\ref{T:mainskew} states that when we restrict $\nu$ to be a partition (regardless of whether it is contained inside $\mu$), we can still apply \allowhookfillings~to generate the determinantal decomposition combinatorially.  That is, 

$$\I_{\mu/\nu}=\sum_{\gamma\in \SHF_{\mu/\nu}} \prod_{\h(r,\c_r)\in \gamma} \ \sign(\h(r,\c_r))\ H_{\Delta(\h(r,\c_r))},$$
where $\SHF_{\mu/\nu}$ denotes the set of all \allowhook~coverings of the \GBPR~diagram $D_{\mu/\nu}.$

Section~\ref{sec:minors} describes how to utilize submatrices to produce the homogeneous function expansion of an immaculate function indexed by a skew shape.  In Section~\ref{sec:skewproof}, we apply these submatrices to prove Theorem~\ref{T:mainskew}.  In Section~\ref{S:nuweakdec}, we discuss how to extend this approach to the situation in which $\nu$ is an arbitrary sequence.  Section~\ref{sec:schur} explains how to recover the decomposition of a Schur function into the complete homogeneous symmetric functions in $\Sym$.

\subsection{Submatrices and immaculates
indexed by skew shapes}{\label{sec:minors}}

Recall that the Jacobi-Trudi formula stated in Theorem \ref{T:NsymJacobiTrudi}~\cite{BBSSZ14}  requires the determinant of $M_\mu$ be computed expanding
using Laplace expansion row by row starting with the first row; the order in which the expansion occurs is important since the $H_a$
are not commutative.
Recalling this convention, we have
$\I_\mu=\ndet(M_{\mu})$ where $(M_{\mu})_{i,j} = H_{\mu_i+j-i}$
and $1 \le i,j \le n.$
Using the permutation expansion of the determinant, we have
\begin{align}
\I_\mu=&\sum_{\sigma\in S_k} \sign(\sigma) (M_{\mu})_{1,\sigma_1} (M_{\mu})_{2,\sigma_2}\cdots (M_{\mu})_{k,\sigma_k}
\cr
=&\sum_{\sigma\in S_k} \sign(\sigma)
H_{\mu_1+\sigma_1-1} H_{\mu_2+\sigma_2-2}\cdots H_{\mu_k+\sigma_k-k},
\end{align}
where $\sign(\sigma)$ is the sign of $\sigma.$
Due to properties of the Jacobi-Trudi matrix,
it is possible to describe the submatrices of $M_\mu$
in terms of permutations.  
Recall that with $\sigma\in S_k$ and $m<k$, the set $\beta_{\sigma,m}$ is the collection $\{\sigma_{m+1},\sigma_{m+2}, \cdots, \sigma_{k}\}$
and $\beta_{\sigma,m,j}$ is the $j^{th}$ smallest element of $\beta_{\sigma,m}$.

Let $M(i|j)$ be the submatrix obtained from the matrix $M$
by deleting  row $i$ and column $j$.
With this notation, the submatrix 
obtained by deleting row $i$ and column $\sigma_i$ for $1 \le i \le m$ is
$$M_\mu^{(m,\sigma)}:=(M_{\mu})(1|\sigma_1)(2|\sigma_2)\ldots(m|\sigma_{m}).$$  Notice that $\beta_{\sigma,m}$ is the set of columns that were not deleted in the construction of $M_\mu^{(m,\sigma)}$.  

In the following proposition, we use notation $\tilde{\mu}$ and $\tilde{\nu}$ to denote the last $k-m$ parts of $\mu$ and $\nu^{(m)}$, respectively.  Although this process of removing the first $m$ parts relies on $m$, we suppress the $m$ in our notation for clarity of exposition. 

\begin{prop}{\label{P:minordet}}
Let $\mu=(\mu_1, \hdots , \mu_k) \in \mathbb{Z}^k$ and let $\sigma=(\sigma_1, \sigma_2, \hdots , \sigma_k)$ be a permutation in $S_{k}$.  Let $\tilde{\mu}=(\mu_{m+1}, \mu_{m+2}, \hdots , \mu_k)$ be the sequence obtained by deleting the first $m$ parts of $\mu$ and let $\tilde{\nu} = (\nu_{m+1}^{(m)}, \nu_{m+2}^{(m)}, \hdots , \nu_k^{(m)})$ be the partition consisting of the last $k-m$ parts of the sequence obtained during the construction of the first $m$ \allowhooks~corresponding to the first $m$ entries in the permutation $\sigma$.  Then $$(M_{\mu}^{(m, \sigma)})_{i,j}=H_{(\tilde{\mu}_{i}-i)-(\tilde{\nu}_{j}-j)}.$$
\end{prop}

\begin{proof}
We use the Laplace expansion
for computing the Jacobi-Trudi determinant.  
After expanding through the first $m$ rows following the ordering given by the permutation $\sigma\in S_k$, the entries of the resulting submatrix
 $M_{\mu}^{(m, \sigma)}$ are given by the formula
\begin{equation}
    (M_{\mu}^{(m, \sigma)})_{i,j}= H_{\mu_{m+i}-(m+i)+\beta_{\sigma,m,j}},
    \label{E:RhoExp}
\end{equation}
with $1 \le i,j \le k-m$.
For $1\le j \le k-m$,  we set 
$$\zeta_{j}^{(m,\sigma)}=m+
j-\beta_{\sigma,m,j}.$$
There are $k-(\beta_{\sigma,m,j})$ entries larger than $\beta_{\sigma,m,j}$
in $\sigma\in S_k$ with exactly $k-m -j$ of them in $\{ \sigma_{m+1},\ldots,\sigma_{k} \}.$  Therefore, the number of entries larger than $\beta_{\sigma,m,j}$ in $\{ \sigma_1, \hdots , \sigma_m\}$ equals $$k-(\beta_{\sigma,m,j}) - (k-m-j) = m+j-(\beta_{\sigma,m,j}) = \zeta_{j}^{(m,\sigma)}.$$

Substituting $m$ for $r$ and $j+m$ for $j$ in Lemma~\ref{T:countnuj} implies $\nu_{j+m}^{(m)}= \zeta_{j}^{(m,\sigma)}$.  Then $\beta_{\sigma,m,j} = m+j-\nu_{m+j}^{(m)}$.  Equation~\eqref{E:RhoExp} now produces $$(M_{\mu}^{(m, \sigma)})_{i,j}= H_{\mu_{m+i}-(m+i)+(m+j)-\nu_{m+j}^{(m)}}=H_{(\mu_{m+i}-i)-(\nu_{m+j}^{(m)}-j)},$$ so that $$(M_{\mu}^{(m, \sigma)})_{i,j}=H_{(\tilde{\mu}_{i}-i)-(\tilde{\nu}_{j}-j)},$$ as desired.
\end{proof}

Proposition~\ref{P:minordet} together with Equation~\eqref{Def:Imunu} imply that the skew immaculate $\I_{\tilde{\mu}/\tilde{\nu}}$ is equal to the determinant of the submatrix $M_{\mu}^{(m,\sigma)}$.  This provides a way to start with a sequence $\mu$ and produce a skew immaculate $\I_{\tilde{\mu}/\tilde{\nu}}$.  In the next section, we discuss how to start with the skew shape, modify it to create a starting un-skewed sequence, and then apply the \allowhookfilling~techniques to construct its $H$-decomposition combinatorially.

\subsection{Proof of Theorem~\ref{T:mainskew}}{\label{sec:skewproof}}

In the following, $\lambda^T$ denotes the transpose of the partition $\lambda$, obtained by reflecting $\lambda$ across the main diagonal (southwest to northeast).

\begin{proof}
[Proof of Theorem~\ref{T:mainskew}]
Let $\mu \in \mathbb{Z}^k$ and let $\lambda$ be a partition with $\ell$ nonzero parts, where $\ell \le k$.  If $\ell <k$, set $\lambda_{i}=0$ for $\ell < i \le k$.  Prepend $\lambda_1$ rows of length $\lambda_1$ to the front of $\mu$ to obtain $\underline{\mu} = (( \lambda_1)^{\lambda_1}, \mu_1, \hdots, \mu_k)$.  Applying Theorem~\ref{T:main} to the composition $\underline{\mu}$ produces a collection of \allowhookfillings~whose signed weights produce the decomposition of $\I_{\underline{\mu}}$ into complete homogeneous functions.  Consider the \allowhookfillings~of $\underline{\mu}$ whose first $\lambda_1$ \allowhooks~are the collection $$\{\h(1,\c_1), \h(2, \c_2), \hdots , \h(\lambda_1,\c_{\lambda_1})\},$$ where $\c_j=(\lambda_1+(\lambda^T)_j,j)$. 
 
 To see that $(\lambda_1+(\lambda^T)_j,j) \in \mathbb{T}_{\underline{\mu}/\nu^{(j-1)}}$ for each $1 \le j \le \lambda_1,$ first note that $\c_1=(\lambda_1+(\lambda^T)_1,1)$ is in $\mathbb{T}_{\mu/\nu^{(0)}}$, since every cell in the leftmost column of $D_{\mu/\nu^{(0)}}$ is in $\mathbb{T}_{\mu/\nu^{(0)}}$.  The \allowhook~$\h(1,\c_1)$ is shaped like the letter L; it is comprised precisely of the cells $$\{ (1,1),(1,2), \hdots , (1,\lambda_1), (2,1), (3,1), \hdots , (\lambda_1+(\lambda^T)_1,1) \}.$$  In particular, the cell $(\lambda_1+(\lambda^T)_2,1)$ is contained in $\h(1,\c_1)$ since $(\lambda^T)_2 \le (\lambda^T)_1$.  But then $(\lambda_1+(\lambda^T)_2-1,1) \in \h(1, \c_1)$ and is not the terminal cell for $\h(1, \c_1)$ since $(\lambda_1+(\lambda^T)_2-1,1)$ lies immediately below $(\lambda_1+(\lambda^T)_2,1)$ and $\h(1,\c_1)$ is L-shaped.  Step~\ref{add11} of Procedure~\ref{Pr:perm} implies that $(\lambda_1+(\lambda^T)_2,2) \in \mathbb{T}_{\mu/\lambda^{(1)}}$ since $(\lambda_1+(\lambda)^T_2,2)$ is obtained from $(\lambda_1+(\lambda^T)_2-1,1)$ by adding $(1,1)$.  
 
 The \allowhook~$\h(2,\c_2)$ is then also shaped like the letter L, consisting of the cells $$\{(2,2),(2,3), \hdots , (2,\lambda_1), (3,2), (4,2), \hdots , (\lambda_1+(\lambda^T)_2,2) \}.$$  We can yet again use the fact that $\lambda^T$ is a partition to conclude that $(\lambda_1+(\lambda^T)_3,2) \in \h(2, \c_2)$.  This implies that $(\lambda_1+(\lambda^T)_3,3) \in \mathbb{T}_{\mu/\nu^{(2)}}$.  Repeating this argument shows that the cells $\{\c_1, \c_2, \hdots , \c_{\lambda_1} \}$ are indeed terminal cells, and so the collection $\{\h(1,\c_1), \h(2, \c_2), \hdots , \h(\lambda_1,\c_{\lambda_1})\}$ is indeed a valid \allowhookfilling.

Next, consider the partial \GBPR~diagram obtained after $\lambda_1$ iterations of Procedure~\ref{A:THC} (Step 2).  The first $\lambda_1$ rows of $\underline{\mu}$ have been removed and the remaining nonzero rows of $\nu^{(\lambda_1)}$ are equal to the rows of $\lambda$, since their columns are the rows of $\lambda^T$.  Therefore, removing the first $\lambda_1$ rows entirely from this diagram produces precisely the \GBPR~diagram $D^{(0)}_{\mu/\lambda}$.

Finally, we prove that the determinant $\ndet\left(M_{\mu/\lambda}\right)$ is equal to the determinant of the submatrix of $M_{\underline{\mu}}$ obtained by removing rows $1$ through $\lambda_1$ and columns $\{\lambda_1+(\lambda^T)_1,\lambda_1+(\lambda^T)_2-1, \hdots , \lambda_1+(\lambda^T)_{\lambda_1}-\lambda_1+1\}.$  Let $\sigma$ be the permutation whose first $\lambda_1$ parts are $(\lambda_1+(\lambda^T)_1,\lambda_1+(\lambda^T)_2-1, \hdots , \lambda_1+(\lambda^T)_{\lambda_1}-\lambda_1+1)$. These parts are obtained by listing the diagonals containing the terminal cells; that is, cell $(\lambda_1+(\lambda^T)_j,j)$ lies in diagonal $\mathcal{L}_{\lambda_1+(\lambda^T)_j-j+1}$.  Proposition~\ref{P:minordet} implies that $$\ndet(M^{(\lambda_1,\sigma)}_{\underline{\mu}}) =\ndet\left((H_{(\mu_i-i)-(\lambda_j-j)})_{i,j}\right)= \I_{\mu/\lambda},$$ as desired, since $\lambda_j=\nu_{\lambda_1+j}^{(\lambda_1)},$ as observed above.
\end{proof}

Note that there are other choices for the prefix of $\mu$ which yield the same result.  This particular choice is taken so that none of the first $\lambda_1$ rows of the matrix $M_{\underline{\mu}}$ contains the entry $H_0$.  In particular, observe that for the \allowhooks~given in the proof of Theorem~\ref{T:mainskew}, $\Delta(\h(j , \c_j)) > 0$ for $1 \le j \le \lambda_1$.  To see this, first note that $\h(j , \c_j)$ contains $\lambda_1-j+1$ blue cells in row $j$.  This means $\spin(j)=\lambda_1-j+1>0$.  Since $\Delta(\h(j , \c_j)) \ge \spin(j)$, we have $\Delta(\h(j , \c_j)) > 0$ for $1 \le j \le \lambda_1$.

This allows us to determine the $H$-decomposition of $\I_{\mu/\lambda}$ by first computing the $H$-decomposition of $\I_{\underline{\mu}}$ and then identifying the terms in this decomposition beginning with $$H_{2 \lambda_1 - 1 + (\lambda^T)_1} H_{2 \lambda_1 - 3 + (\lambda^T)_2} \cdots H_{2 \lambda_1 - (2 \lambda_1 - 1) + (\lambda^T)_{\lambda_1}}$$ (and deleting all other terms).  Removing this initial product from each of the identified terms and then summing the resulting terms produces the $H$-decomposition of $\I_{\mu/\lambda}$.

\begin{ex}
Let $\mu=(5,-1,3,4)$ and $\lambda=(3,1,0,0)$.  Then $\ell=2$ and $\lambda^T=(2,1,1,0)$.  Therefore, $\underline{\mu}=(3,3,3,5,-1,3,4)$ and the initial \GBPR~diagram for $\underline{\mu}$ is 

\hskip 1.05truein
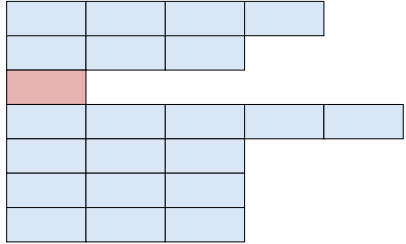
\begin{picture}(300,100)
\put(0,0){
\hbox{
\vbox{
\Yfillcolour{lightblue}
\young(\ \ \ \ )
\vskip -1.295pt
\Yfillcolour{lightblue}
\young(\ \ \ )
\vskip -1.295pt
\Yfillcolour{lightred}
\young(\ )
\vskip -1.295pt
\Yfillcolour{lightblue}
\young(\ \ \ \ \ )
\vskip -1.295pt
\Yfillcolour{lightblue}
\young(\  \ \ )
\vskip -1.295pt
\Yfillcolour{lightblue}
\young(\  \ \ )
\vskip -1.295pt
\Yfillcolour{lightblue}
\young(\  \ \ )
}}
}
\end{picture}

\vspace*{.1in}

The first $3$ (since $\lambda_1=3$) \allowhooks~terminate at the \allowcells~
$\{(5,1),(4,2),(4,3)\}$ (since $\lambda_1+(\lambda^T)_1=3+2=5, \lambda_1+(\lambda^T)_2=3+1=4,$ and $\lambda_1+(\lambda^T)_3=3+1=4$), producing the partial \allowhookfilling

\hskip 1.05truein
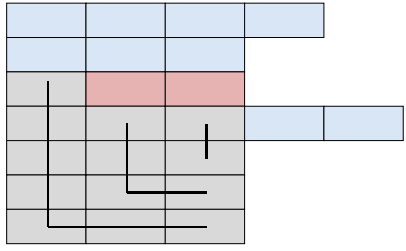
\begin{picture}(300,100)
\put(0,0){
\hbox{
\vbox{
\Yfillcolour{lightblue}
\young(\ \ \ \ )
\vskip -1.295pt
\Yfillcolour{lightblue}
\young(\ \ \ )
\vskip -1.295pt
\young(!<\Yfillcolour{grey}>\  !<\Yfillcolour{lightred}>\ \  )
\vskip -1.295pt
\young(!<\Yfillcolour{grey}>\ \ \ !<\Yfillcolour{lightblue}>\ \ )
\vskip -1.295pt
\Yfillcolour{grey}
\young(\  \ \  )
\vskip -1.295pt
\Yfillcolour{grey}
\young(\  \ \ )
\vskip -1.295pt
\Yfillcolour{grey}
\young(\  \ \ )}}}

\put(35,3){\line(1,0){60}}
\put(35,3){\line(0,1){55}}

\put(65,16){\line(1,0){30}}
\put(65,16){\line(0,1){26}}

\put(95,29){\line(0,1){13}}

\end{picture}

\vspace*{.1in}

The skew immaculate $\I_{(5,-1,3,4)/(3,1,0,0)}$ can be obtained by selecting the terms appearing in $\I_{(3,3,3,5,-1,3,4)}$ whose first three terms are $H_{7} H_{4} H_{2}$.

Notice that the submatrix $M_{\underline{\mu}}(1|5)(2|3)(3|2)$ of the matrix $$M_{\underline{\mu}} = \begin{bmatrix} H_3 &H_4 & H_5 & H_6 & H_7 & H_8 & H_9   \\ H_2 & H_3 & H_4 & H_5 & H_6 & H_7 & H_8 \\
H_1 & H_2 & H_3 & H_4 & H_5 & H_6 & H_7  \\
H_2 & H_3 & H_4 & H_5 & H_6 & H_7 & H_8 \\ H_{-5} & H_{-4} & H_{-3} & H_{-2} & H_{-1} & H_0 & H_1 \\ H_{-2} & H_{-1} & H_0 & H_1 & H_2 & H_3 & H_4 \\ H_{-2} & H_{-1} & H_0 & H_1 & H_2 & H_3 & H_4 \end{bmatrix}$$ (obtained by deleting the first through third rows and columns $5,3,$ and $2$) is given by $$\begin{bmatrix} H_2 & H_5 & H_7 & H_8 \\ H_{-5} & H_{-2} & H_0 & H_1 \\ H_{-2} & H_1 & H_3 & H_4 \\ H_{-2} & H_1 & H_3 & H_4 \end{bmatrix}$$ which is precisely the matrix defined by $(M_{\mu/\lambda})_{i,j} = H_{(\mu_i-i)-(\lambda_j-j)}$
\end{ex}

\subsection{Skewing by non-partition shapes}{\label{S:nuweakdec}}

First assume that $\lambda$ is an integer sequence, all of whose parts are nonnegative.  As in the analogous $\Sym$ situation, if we set  
$$\hat \lambda=(\lambda_1,\lambda_2,\ldots,\lambda_{p-1},\lambda_{p+1}-1,\lambda_{p}+1,\lambda_{p+2},\ldots,\lambda_k)$$
and $(M_{\mu/\hat{\lambda}})_{i,j}=H_{{(\mu_i-i)-(\hat\lambda_j-j)}},$
then
$(M_{\mu/\lambda})_{i,j}=(M_{\mu/\hat{\lambda}})_{i,j}$ for $j \notin \{p, p+1\}.$
Furthermore, since $\hat\lambda_{p+1}=\lambda_p+1$
and $\hat\lambda_p=\lambda_{p+1}-1$, we have
$$(M_{\mu/\lambda})_{i,p}=H_{(\mu_i-i)-(\lambda_p-p)}=H_{{(\mu_i-i)-(\hat\lambda_{p+1}-p-1)}}=(M_{\mu/\hat{\lambda}})_{i,p+1}$$
and 
$$(M_{\mu/\lambda})_{i,p+1}=H_{(\mu_i-i)-(\lambda_{p+1}-(p+1))}=H_{{(\mu_i-i)-(\hat\lambda_{p}+1-(p+1))}}=(M_{\mu/\hat{\lambda}})_{i,p}$$ for all $1 \le i \le k$.
Thus, $\I_{\mu/\lambda}=-\I_{\mu/\hat \lambda},$ since $M_{\mu/\hat{\lambda}}$ is obtained from $M_{\mu/\lambda}$ by swapping columns $p$ and $p+1$.

Note that if $\lambda_{i+1}=\lambda_i+1$, for some $1 \le i < k$,
then columns $i$ and $i+1$ of $M_{\mu/\lambda}$ are identical.  Thus the row by row expansion of $\ndet(M_{\mu/\lambda})$ starting in the first row and continuing down yields that
$\I_{\mu/\lambda}=\ndet(M_{\mu/\lambda})=0.$

When taken together, the results of the two
previous paragraphs imply that, without loss of generality, we may assume that the entries of $\lambda$ are weakly decreasing.  If the entries of $\lambda$ are not weakly decreasing apply the straightening operator $\hat{\lambda}$ (adjusting the sign each time) until the result is a partition (in which case Theorem~\ref{T:mainskew} applies) or a sequence containing a one-step increase (in which case the skew immaculate function is zero).

Finally, if any terms of $\lambda$ are negative, let $\lambda_j$ be the smallest part of $\lambda$.  Add $-\lambda_j$ to every part of $\mu$ and every part of $\lambda$; call the resulting sequences $\aug_{\lambda_j}(\mu)$ and $\aug_{\lambda_j}(\lambda)$ respectively.  Then $M_{\aug_{\lambda_j}(\mu) / \aug_{\lambda_j}(\lambda)} = M_{\mu/\lambda}$, so 
\begin{equation}
\label{E:Getridofnegatives}
\I_{\mu/\lambda}=\I_{\aug_{\lambda_j}(\mu) / \aug_{\lambda_j}(\lambda)}
\end{equation}
and we can apply the above techniques to $\I_{\aug_{\lambda_j}(\mu) / \aug_{\lambda_j}(\lambda)}$ to find the decomposition of $\I_{\mu/\lambda}$ into the complete homogeneous noncommutative symmetric functions.

\begin{ex}
Let $\mu=(2,-5,0,1)$ and $\lambda=(2,-3,1,6)$.  To determine $\I_{\mu/\lambda}$ combinatorially, we must apply several steps before finding the tunnel hook fillings.
\begin{enumerate}
\item Since $\lambda$ includes negative parts (with smallest part equal to $-3$), add $3$ to every part of $\lambda$ and every part of $\mu$ to get $$\aug_3(\mu)/\aug_3(\lambda) = (5,-2,3,4)/(5,0,4,9).$$
\item Since $\aug_3(\lambda)$ is not a partition, apply the straightening operator to get $$\I_{\mu/\lambda} = -\I_{(5,-2,3,4)/(5,0,8,5)} = \I_{(5,-2,3,4)/(5,7,1,5)} $$ $$= -\I_{(5,-2,3,4)/(6,6,1,5)} = \I_{(5,-2,3,4)/(6,6,4,2)}.$$
\end{enumerate}
\end{ex}

The following corollary (an immediate consequence of Proposition~\ref{P:minordet} and Proposition~\ref{T:perm}) provides a way to expand the immaculates $\I_\mu$ in terms of skew immaculates.  A \emph{linear permutation} (or \emph{partial permutation}) of a $k$-element set is an ordered arrangement of an $m$-element subset of a $k$-element set.  (Note that in the literature, these are also sometimes referred to as partial permutations or $k$-permutations.)

Let $A_{k,m}$ be the set of all linear permutations of an $m$-element subset of a $k$-element set.  Define the \emph{sign} $\sign(\pi)$ of a linear permutation $\pi=(\pi_1, \hdots , \pi_m)$ to be $(-1)^{\delta}$, where $\delta$ is the number of pairs $a,b$ such that either $a=\pi_i> \pi_j=b$ with $i<j$ or $a=\pi_i >q$ with $q$ a positive integer not equal to $\pi_j$ for any $j$.

\begin{cor}
Let $\mu \in \mathbb{Z}^k$ and $m$ be a fixed integer such that $1 \le m \le k$.  If $\tilde{\mu}$ is the sequence obtained by deleting the first $m$ parts of $\mu$, then
$$
\I_\mu=
\sum_{\pi
\in A_{k,m}}
\sign(\pi)
\left( \prod_{1 \le i \le m} H_{\mu_i-i+\pi_i} \right)
\I_{\tilde{\mu}/\nu^{(m)}},
$$

\noindent
where $\nu^{(m)}$ is the sequence obtained from the construction of the \allowhooks~corresponding to $\pi$ in the diagram $D_{\mu}$.
\end{cor}

For example, when $\mu=(4,3,3,2)$ and $m=2$, we have the following decomposition of $\I_{(4,3,3,2)}$.
\begin{align*}
\I_{(4,3,3,2)}=\
&H_{(4,3)}\I_{(3,2)}
- H_{(4,4)}\I_{(3,2)/(1,0)}
+ H_{(4,5)}\I_{(3,2)/(1,1)}\\
-\ &H_{(5,2)}\I_{(3,2)}
+ H_{(5,4)}\I_{(3,2)/(2,0)}
- H_{(5,5)}\I_{(3,2)/(2,1)}\\
+\ &H_{(6,2)}\I_{(3,2)/(1,0)}
- H_{(6,3)}\I_{(3,2)/(2,0)}
+ H_{(6,5)}\I_{(3,2)/(2,2)}\\
-\ &H_{(7,2)}\I_{(3,2)/(1,1)}
+ H_{(7,3)}\I_{(3,2)/(2,1)}
- H_{(7,4)}\I_{(3,2)/(2,2)}
\end{align*}

Note again that applying the forgetful map to this expansion produces the expansion of a Schur function in terms of skew Schur functions with complete homogeneous symmetric functions as coefficients.  Thus we have the following Schur function decomposition.
\begin{align*}
s_{(4,3,3,2)}=\ 
&h_{(4,3)}s_{(3,2)}
- h_{(4,4)}s_{(3,2)/(1)}
+ h_{(4,5)}s_{(3,2)/(1,1)}\\
-\ &h_{(5,2)}s_{(3,2)}
+ h_{(5,4)}s_{(3,2)/(2)}
- h_{(5,5)}s_{(3,2)/(2,1)}\\
+\ &h_{(6,2)}s_{(3,2)/(1)}
- h_{(6,3)}s_{(3,2)/(2)}
+ h_{(6,5)}s_{(3,2)/(2,2)}\\
-\ &h_{(7,2)}s_{(3,2)/(1,1)}
+ h_{(7,3)}s_{(3,2)/(2,1)}
- h_{(7,4)}s_{(3,2)/(2,2)}
\end{align*}

\subsection{Recovering the Schur function decomposition}{\label{sec:schur}}

The forgetful map applied to immaculate functions produces the Schur functions, and the forgetful map applied to the $H$ basis for $\NSym$ produces the $h$ basis for $\Sym$.  Therefore, applying the forgetful map to the decomposition of $\I_{\mu/\lambda}$ into the $H$ basis (where $\mu$ is a partition and $\lambda$ is a partition such that $\lambda_i \le \mu_i$ for all $i$) produces the decomposition of the skew Schur function $s_{\mu/\lambda}$ into the $h$ basis.

Our approach, therefore, provides an alternative to the special rim hooks appearing in the E\u{g}ecio\u{g}lu-Remmel~\cite{EgeRem90} combinatorial interpretation of the inverse Kostka matrix.  While \allowhooks~and special rim hooks share some similarities, they are not simply shifts or translates of one another.  

Rim hooks are collections of cells in the diagram of a partition satisfying the following properties.
\begin{enumerate}
\item Rim hooks consist of cells on the northeastern rim of the diagram.
\item The cells in a rim hook are connected.
\item A rim hook contains no $2 \times 2$ squares.
\item A rim hook is \emph{special} if every rim hook includes a cell in the leftmost column.
\end{enumerate}

If one attempts to apply rim hooks to immaculates indexed by shapes which are not partitions, 
one must sacrifice either property (1) or (2).  Example~\ref{E:issues} illustrates this obstruction, since the shape $(3,8,4,1)$ must admit a special rim hook filling whose lengths are $6,0,8,2$ respectively.  In order to have a special rim hook of length $8$ beginning in the third row, there cannot also be a connected special rim hook of length $6$ beginning in the top row.

\begin{ex}{\label{E:issues}}
Let $\mu=(3,8,4,1)$ and consider the term $H_{(6,0,8,2)}$.  To fill the shape $(3,8,4,1)$ with special rim hooks, a special rim hook of length $6$ starting in the top row is necessary, but this is not compatible with a special rim hook of length $8$ starting in the third row from the top.

\vskip 5pt
\hskip .425truein
\hbox{
\vbox{
\Yfillcolour{lightblue}
\young(\  )
\vskip -1.295pt
\Yfillcolour{lightblue}
\young(\ \ \ \ )
\vskip -1.295pt
\Yfillcolour{lightblue}
\young(\ \ \ \ \ \ \ \ )
\vskip -1.295pt
\Yfillcolour{lightblue}
\young(\ \ \  )
}
}
\end{ex}

If one relaxes the rules for special rim hooks to try to adress this concern, other lengths (such as an initial rim hook of length $7$) become available which are not legal options since they do not appear as indices in the $H$-decomposition of $\I_{(6,0,8,2)}$.  In fact, Loehr and Niese point out that their special rim hook method (which solely applies to immaculates indexed by partitions) does not compute the content of a diagram simply by listing the lengths in some predetermined order~\cite{LoeNie21}.  


In order to generalize the E\u{g}ecio\u{g}lu-Remmel decomposition idea to composition shapes (for $\NSym$), we need objects that ``tunnel" into the interior of the diagram.  Our \allowhooks~satisfy properties (2) and (3) above, as well as a variant of (1) stating that \allowhooks~are comprised of cells on the South-Western border of the diagram.  See Example~\ref{E:rimhooks} to compare a tunnel hook covering of skew shape $(7,7,7,6,6,4)/(2,2,1,1)$ and the corresponding special rim hook tableau.

\begin{ex}{\label{E:rimhooks}}
Let $\mu=(7,7,7,6,6,4)$ and $\lambda=(2,2,1,1)$.  The \allowhookfilling~which produces the term $H_{(8,10,4,8,0,1)}$ is shown.  

\hskip .6truein
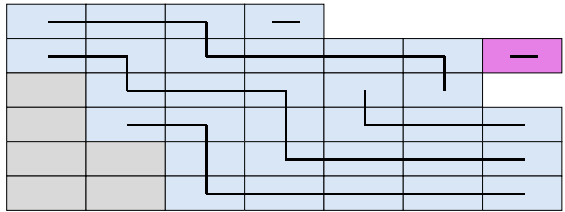
\begin{picture}(300,80)
\put(0,0){
\hbox{
\vbox{
\Yfillcolour{lightblue}
\young(\ \ \ \ )
\vskip -1.295pt
\Yfillcolour{lightblue}
\young(\ \ \ \ \ \ !<\Yfillcolour{lightpurple}>\ )
\vskip -1.295pt
\Yfillcolour{grey}
\young(\  !<\Yfillcolour{lightblue}>\ \ \ \ \ )
\vskip -1.295pt
\Yfillcolour{grey}
\young(\  !<\Yfillcolour{lightblue}>\ \ \ \ \ \ )
\vskip -1.295pt
\Yfillcolour{grey}
\young(\ \ !<\Yfillcolour{lightblue}>\ \ \ \ \ )
\vskip -1.295pt
\Yfillcolour{grey}
\young(\ \ !<\Yfillcolour{lightblue}>\ \ \ \ \ )
}}}

\put(95,3){\line(1,0){120}}
\put(95,3){\line(0,1){26}}
\put(65,29){\line(1,0){30}}

\put(125,16){\line(1,0){90}}
\put(125,16){\line(0,1){26}}
\put(65,42){\line(1,0){60}}
\put(65,42){\line(0,1){13}}
\put(35,55){\line(1,0){30}}

\put(155,29){\line(1,0){60}}
\put(155,29){\line(0,1){13}}

\put(185,42){\line(0,1){13}}
\put(95,55){\line(1,0){90}}
\put(95,55){\line(0,1){13}}
\put(35,68){\line(1,0){60}}

\put(210,55){\line(1,0){10}}

\put(120,68){\line(1,0){10}}

\end{picture}

\vskip 2pt

\noindent The corresponding special rim hook tableau is shown below.  Note that this produces the term $h_{(10,8,8,4,1,0)}$ due to the commutativity of $\Sym$.  There are no other rim hook tableaux of this shape corresponding to $h_{(10,8,8,4,1,0)}$.

\hskip .65truein
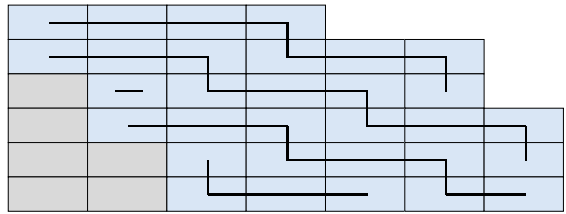
\begin{picture}(300,80)
\put(0,0){
\hbox{
\vbox{
\Yfillcolour{lightblue}
\young(\ \ \ \ )
\vskip -1.295pt
\Yfillcolour{lightblue}
\young(\ \ \ \ \ \ )
\vskip -1.295pt
\Yfillcolour{grey}
\young(\  !<\Yfillcolour{lightblue}>\ \ \ \ \ )
\vskip -1.295pt
\Yfillcolour{grey}
\young(\  !<\Yfillcolour{lightblue}>\ \ \ \ \ \ )
\vskip -1.295pt
\Yfillcolour{grey}
\young(\ \ !<\Yfillcolour{lightblue}>\ \ \ \ \ )
\vskip -1.295pt
\Yfillcolour{grey}
\young(\ \ !<\Yfillcolour{lightblue}>\ \ \ \ \ )
}}}

\put(95,3){\line(1,0){60}}
\put(95,3){\line(0,1){13}}

\put(125,16){\line(1,0){60}}
\put(185,3){\line(0,1){13}}
\put(185,3){\line(1,0){30}}
\put(125,16){\line(0,1){13}}
\put(65,29){\line(1,0){60}}

\put(155,29){\line(1,0){60}}
\put(155,29){\line(0,1){13}}
\put(215,16){\line(0,1){13}}
\put(95,42){\line(1,0){60}}
\put(95,42){\line(0,1){13}}
\put(35,55){\line(1,0){60}}

\put(60,42){\line(1,0){10}}

\put(185,42){\line(0,1){13}}
\put(125,55){\line(1,0){60}}
\put(125,55){\line(0,1){13}}
\put(35,68){\line(1,0){90}}

\end{picture}

\end{ex}

Our \allowhook~approach also allows us to generalize to ribbon decompositions of certain immaculate functions, improving upon results of Campbell~\cite{Cam17}.

\subsection{Schur functions in noncommuting variables}

This subsection uses  \cite{ALvW22} for foundational material and notation.
The space of symmetric functions with noncommuting variables (NCSYM) has analogues for the common
symmetric functions in SYM.
Source skew functions in noncommuting variables $s_{[\lambda/\mu]}$ satisfy a variation of the Jacobi-Trudi identity:
\begin{equation}
 s_{[\mu/\nu]}=\mathfrak{det}\left(
 \frac{1}{(\mu_i-i-(\nu_j-j))}h_{[\mu_i-i-(\nu_j-j)]}\right)
\end{equation}
where $h_{[a]}$ is a noncommuting analogue of
the complete homogeneous symmetric function
$h_a$ and $\mu$ and $\nu$ are integers associated
with \textit{set partitions}.
Thus, we can use Theorem \ref{T:mainskew}
to give a combinatorial interpretation of the
corresponding inverse Kostka matrix.

Note that the definition of the NCSYM skew Schur functions is
$\delta \odot s_{[\lambda/\mu]}$ where $\delta \in \mathfrak{S}_n$.  Note that $\delta$ acts on the 
set partition $[\lambda/\mu]$ but it is not always true that $\delta \odot [\lambda/\mu]= \delta ([\lambda/\mu])$ (\cite{ALvW22}, Corollary 4.8).

\section{Ribbon decompositions of immaculate functions}{\label{sec:ribbon}}

The ribbon basis for $\NSym$ is another generalization of Schur functions to $\NSym$.  There are a number of excellent sources for background on ribbons in $\Sym$ and $\NSym$~\cite{GKLLRT95,Mac95}.  Our work in this section extends results of Campbell~\cite{Cam17}.

  We take the following formula as our definition for the ribbon basis.

 \begin{defn}
 The \emph{ribbon basis} for $\NSym$ is given by the formula $$R_{\alpha} = \sum_{\beta \succeq \alpha} (-1)^{\ell(\beta)-\ell(\alpha)}H_{\beta},$$ where $\succeq$ is the refinement order on compositions and where $\ell(\alpha)$ is the length (i.e. number of parts) of $\alpha$.   Here
for $a \in \mathbb{Z}$, $R_a=H_a$ and $R_0=H_0=1$.
\end{defn}  

We will also make use of the formula that converts a complete homogeneous function in $\NSym$ indexed by a strong composition into the ribbon basis for $\NSym$~\cite{GKLLRT95}).  That is, $$H_{\alpha} = \sum_{\beta \succeq \alpha} R_{\beta}.$$

 Multiplication in the $H$ basis is fairly straightforward; simply concatenate the indexing compositions.  Multiplication in the ribbon basis requires two steps, which are described below.

\begin{defn}{\cite{GKLLRT95}}
Let $\alpha=(\alpha_1,\alpha_2,\cdots,\alpha_k)$
and 
$\beta=(\beta_1,\beta_2,\cdots, \beta_j)$ be integer sequences.  Then the concatenation of $\alpha$ and 
$\beta,$ denoted by $\alpha \cdot \beta,$
is given by
\begin{equation*}
    \alpha\cdot \beta=
    (\alpha_1,\alpha_2,\cdots,\alpha_k,\beta_1,\beta_2,\cdots, \beta_j).
\end{equation*}
 Define \emph{near concatenation} $\odot$ by
 \begin{equation*}
    \alpha\odot \beta=
    {(\alpha_1,\alpha_2,\cdots,\alpha_k+\beta_1,\beta_2,\cdots, \beta_j)}.
\end{equation*}
 The product of two ribbons, $R_{\alpha}$ and $R_{\beta}$ (where $\alpha$ and $\beta$ are compositions), is given by
 \begin{equation}
     R_{\alpha} R_{\beta}=
      R_{\alpha\cdot\beta}+R_{\alpha\odot\beta}.
    \label{E:Rstar}
 \end{equation}
 \end{defn}

We now describe a method for expanding the \emph{refinement} order to weak compositions compositions, by first defining a coarsening of a sequence (generalizing the notion of coarsening compositions).

\begin{defn}
Let $\alpha=(\alpha_1, \alpha_2, \hdots , \alpha_k)$ be an integer sequence.  Given a subset $S=\{i_1, i_2, \hdots i_t \}$ of $[k-1]$, the \emph{coarsening $\alpha_S$ of $\alpha$ with respect to $S$} is given by
$$\Theta(\alpha,S)=(\alpha_1 \star_1 \alpha_2 \star_2 \alpha_3 \star_3 \cdots \star_{k-2} \alpha_{k-1} \star_{k-1} \alpha_k)$$ where

\begin{equation}
  \star_i =\begin{cases}
    + & \text{if $i \in S$}.\\
    , & \text{otherwise}.
  \end{cases}
\end{equation}

The \emph{refinement} ordering on weak compositions is given by $\beta \succeq \alpha$ if and only if $\beta$ is a \emph{coarsening} of $\alpha$.  If $\beta \succeq \alpha$, we say that $\alpha$ is a \emph{refinement} of $\beta$. 
\end{defn}

For example, let $\alpha=(5,2,1,4,3,3,2,6,2,3)$ and let $S=\{2,3,5,8\}$.  Then the coarsening of $\alpha$ with respect to $S$ is $(5 , 2 + 1 + 4 , 3 + 3 , 2 , 6 + 2 , 3)=(5,7,6,2,8,3)$.  Therefore $(5,7,6,2,8,3) \succeq (5,2,1,4,3,3,2,6,2,3)$.  Note that as long as $\alpha$ is a strong composition (meaning no parts of $\alpha$ are equal to zero), there is a unique subset producing each distinct coarsening.

When we expand an immaculate function indexed by a strong composition into the homogeneous basis for $\NSym$, some of the terms are indexed by weak compositions.  Since $H_0=1$, we usually simply delete the zeros and consider the indices to be shorter strong compositions.  However, the proofs in this section are aided by keeping track of the positions of the zeros in the indexing compositions.  If $\alpha=(\alpha_1, \alpha_2, \hdots  , \alpha_k)$ is a weak composition whose nonzero parts are $\alpha_{j_1}, \alpha_{j_2}, \hdots , \alpha_{j_{s}}$ (listed in increasing order), let $\fl(\alpha)$ be the composition $(\alpha_{j_1}, \alpha_{j_2}, \hdots , \alpha_{j_{s}})$.  It will be useful to think of $\fl(\alpha)$ as a coarsening; to do so, we standardize notation for the coarsening set.

In the following, we assume the first part of the weak composition is nonzero.  All of the compositions appearing in this section will begin with a nonzero part since this section applies specifically to immaculate functions indexed by compositions rather than weak compositions or sequences.

\begin{defn}
Let $\Delta=(\Delta_1, \hdots , \Delta_{\ell})$ be a weak composition with $\Delta_1 >0$ and let the set $\{i_1, \hdots , i_j\}$ be the collection of indices whose corresponding part is $0$.  Let $Z=\{i_1-1, \hdots , i_j-1\} \subseteq [\ell-1]$.  Then an \emph{allowable flat coarsening subset} for $\Delta$ is any subset of $[\ell-1]$ containing $Z$.  
\end{defn}

For example, let $\alpha=(5,0,3,0,1,5,0,4)$ and consider the strong composition $(5,3,1,9)$ obtained by flattening and coarsening $\alpha$.  The only allowable flat coarsening subset producing $(5,3,1,9)$ is $S=\{1,3,6,7\}$.    Although the coarsening of $\alpha$ with respect to the set $Q=\{ 2,3,6,7\}$ also produces the strong composition $(5,0+3+0,1,5+0+4)=(5,3,1,9)$, $Q$ is not an allowable flat coarsening subset since $\alpha_2=0$ but $2-1 \notin Q$.

Notice that for any composition $\alpha$ and any allowable flat coarsening subset $S$, the coarsening $\Theta(\alpha,S)$ will be a strong composition.  The main reason for introducing the allowable flat coarsening subset terminology is to provide a unique way to represent the coarsening of a weak composition.

\begin{lem}
Given a weak composition $\Delta=(\Delta_1, \hdots , \Delta_{\ell})$ such that $\Delta_1 \not= 0$, and a coarsening $\Delta'$ of $\fl(\Delta)$, there is a unique allowable flat coarsening subset $S$ such that $\Theta(\Delta,S)=\Delta'$.
\end{lem}

\begin{proof}

Suppose $S$ and $\hat S$ are two allowable flat coarsening subsets such that $\Theta(\Delta,S)=\Delta' = \Theta(\Delta,\hat{S})$.  Let $s$ be the smallest integer that is in exactly one of $S$ and $\hat S$.  Without loss of generality, assume that $s \in S.$  Furthermore, we have  $\Delta_{s+1}>0,$ since if $\Delta_{s+1}=0$ then $s$ must be a member of any allowable flat coarsening subset for $\Delta$.  Now, let $s'$  be the smallest of the set of consecutive positive integers in $S$ that ends with $s$ (producing the $i^{th}$ part of the coarsening of $\Delta$ with respect to $S$).  Then
$$
\Delta_{s'}+\Delta_{s'+1}+
\cdots +\Delta_{s}
<\Delta_{s'}+\Delta_{s'+1}+
\cdots +\Delta_{s}+\Delta_{s+1}.
$$
Since $\Delta_{s'}+\Delta_{s'+1}+ \cdots +\Delta_{s}$ is $i^{th}$ part of $\Theta(\Delta,\hat{S})$ while $\Delta_{s'}+\Delta_{s'+1}+ \cdots +\Delta_{s}+\Delta_{s+1}$ is the $i^{th}$ part of $\Theta(\Delta,S)$, it is not possible for these two coarsenings to be the same.  Therefore there must be a unique allowable flat coarsening subset $S$ such that $\Theta(\Delta,S)=\Delta'$.
\end{proof}

\subsection{Ribbon expansions of immaculate functions}

The ribbon basis expands positively into the immaculate basis via standard immaculate tableaux~\cite{BBSSZ14}, but the expansion of the immaculate basis into the ribbon basis is only known for certain special cases.  In particular, Campbell provides the following formulas; one for the ribbon expansion of immaculate functions indexed by rectangles and one for immaculate functions indexed by products of two rectangles satisfying certain size conditions.  

\begin{thm}{\cite{Cam17}}
The ribbon expansion of an immaculate function indexed by a rectangle $m^n$ is given by $$\I_{m^n} = \sum_{\sigma \in S_n} \sign(\sigma) R_{(m-1+\sigma_1, m-2+\sigma_2, \hdots , m-n+\sigma_n)},$$ with the convention that $R_{\alpha}$ vanishes if $\alpha$ contains any nonpositive parts.
\end{thm}

\begin{thm}{\cite{Cam17}}
The ribbon expansion of an immaculate function indexed by the product $\alpha=(a^b, c^d)$ of rectangles satisfying $b \le c$ and $b \le a$ is given by $$\I_{\alpha} = \sum_{\sigma \in S_{k}} \sign(\sigma) R_{(\alpha_1-1+\sigma_1, \alpha_2-2+\sigma_2, \hdots , \alpha_{k}-k+\sigma_{k})},$$ with the convention that $R_{\alpha}$ vanishes if $\alpha$ contains any nonpositive parts.
\end{thm}

It is not true in general that 
\begin{equation}\label{E:fullsize}
\I_{\alpha} = \sum_{\sigma \in S_{k}} \sign(\sigma) R_{(\alpha_1-1+\sigma_1, \alpha_2-2+\sigma_2, \hdots , \alpha_{k}-k+\sigma_{k})}. 
\end{equation} 
One open question is to classify the compositions for which Equation~\eqref{E:fullsize} is true.  In this section, we provide a partial solution to this question by developing a large class of compositions for which this decomposition applies.

Let $\alpha$ be a composition.  We describe a function $f_i$ from a tunnel hook covering $\gamma$ of $\alpha$ to a tunnel hook covering $f_i(\gamma)$ of $\alpha$.  Recall that $s_i$ is the transposition $(i,i+1)$ which swaps $\sigma_i$ and $\sigma_{i+1}$ in the permutation $\sigma=\sigma_1 \cdots \sigma_k$ (using one-line notation).

\begin{defn}{\label{D:involution}}
Let $\mu$ be an arbitrary composition and $\gamma \in \SHF_{\mu}$.  Let $\sigma$ be the permutation associated to $\gamma$.  Set $f_i(\gamma)$ to be the tunnel hook covering of $\mu$ whose associated permutation is $s_i (\sigma)$.  If $\Delta$ is the sequence associated to the tunnel hook covering $\gamma$, then let $f_i(\Delta)$ denote the sequence associated to the tunnel hook covering $f_i(\gamma)$.  
\end{defn}

It is immediate that $f_i$ is an involution, since $s_i(s_i(\sigma))=\sigma$.  We now describe how applying $f_i$ to a tunnel hook covering impacts the tunnel hooks.  Note that Lemma~\ref{lem:notin} implies that if $(p_i,q_i)$ and $(p_{i+1},q_{i+1})$ are distinct terminal cells of a tunnel hook  covering, then $p_i-q_i \not= p_{i+1}-q_{i+1}$.  Therefore we only need to consider the cases $p_i-q_i < p_{i+1}-q_{i+1}$ and $p_i-q_i > p_{i+1}-q_{i+1}$. 

\begin{prop}{\label{prop:tailswap}}
Let $\mu$ be an arbitrary composition and $\gamma \in \SHF_{\mu}$.  Let $\h(1,\c_1), \\ \hdots, \h(k,\c_k)$ be the tunnel hooks in $\gamma$ with $\c_j=(p_j,q_j).$  Then the terminal cells of $f_i(\gamma)$ are 
$$
\begin{cases}
(\c_1,\ldots,\c_{i-1},\c_{i+1},(p_i+1,q_i+1),\c_{i+2},\ldots,\c_k)&\hbox{if $p_i-q_i<p_{i+1}-q_{i+1},$}\\
(\c_1,\ldots,\c_{i-1},(p_{i+1}-1,q_{i+1}-1),\c_i,\c_{i+2},\ldots,\c_k)&\hbox{if $p_i-q_i>p_{i+1}-q_{i+1}$}.\end{cases}$$
\end{prop}

\begin{proof}
Assume $(\c_1, \hdots , \c_{i-1}, (p_i,q_i),(p_{i+1},q_{i+1}), \c_{i+2}, \hdots , \c_k)$ are the terminal cells for $\gamma$ with associated permutation $\sigma$.  Recall Proposition~\ref{T:perm} implies that $\sigma_i=p_i-q_i+1$.  First assume $p_i-q_i < p_{i+1}-q_{i+1}$ (which also means $\sigma_i<\sigma_{i+1}$).  Then applying $s_i$ to the associated permutation $\sigma$ corresponds to selecting terminal cell $(p_{i+1},q_{i+1})$ for the $i^{th}$ tunnel hook.  In this case the $i^{th}$ tunnel hook now covers cell $(p_i,q_i)$ and therefore by Lemma~\ref{T:Losecells}(C), $(p_i+1, q_i+1)$ becomes the terminal cell for the tunnel hook corresponding to $(f_i(\sigma))_{i+1} = p_i-q_i+1$ (writing $f_i(\sigma)$ to denote the permutation associated to $f_i(\gamma)$).

If $p_i-q_i > p_{i+1}-q_{i+1}$, then $f_i(\gamma)$ has permutation $f_i(\sigma)$ with $(f_i(\sigma))_i < (f_i(\sigma))_{i+1}$.  Therefore applying the involution $f_i$ to $\gamma$ effectively ``undoes" the operation in the first situation.  So the terminal cell for the $i^{th}$ tunnel hook is $(p_{i+1}-1,q_{i+1}-1)$ and the terminal cell for the $(i+1)^{th}$ tunnel hook is $(p_i,q_i)$, as desired.
\end{proof}


\begin{lem}{\label{lem:signs}}
$$- \prod_{\h(r,\c_r)\in \gamma} \ \sign(\h(r,\c_r))= \prod_{\h(r,\c_r)\in f_i(\gamma)} \ \sign(\h(r,\c_r))$$
\end{lem}

\begin{proof}
The non-commutative determinant of $M_{\mu}$ can be computed according to the formula \begin{equation}{\label{eq:detperm}}
\ndet(M_{\mu})=\sum_{\sigma \in S_k} \sign(\sigma) (M_{\mu})_{1,\sigma_1}(M_{\mu})_{2,\sigma_2} \cdots (M_{\mu})_{k, \sigma_k}.
\end{equation}
Each tunnel hook covering corresponds to one of the terms in Equation~(\ref{eq:detperm}).  The map $f_i$ replaces the tunnel hook corresponding to $\sigma$ to the tunnel hook corresponding to $s_i(\sigma)$, which multiplies the sign by $-1$.
\end{proof}

\begin{theorem}{\label{thm:diagonal}}
Let $\alpha=(\alpha_1, \hdots , \alpha_k)$ be any composition such that $\alpha_i \ge i$ for $1 \le i \le k$. Then $$\I_{\alpha} = \sum_{ \sigma \in S_k} \sign(\sigma) R_{(\alpha_1-1+\sigma_1, \alpha_2-2+\sigma_2, \hdots , \alpha_k-k+\sigma_k)}.$$
\label{T:alpha_i_big}
\end{theorem}

\begin{proof}
Recall the complete homogeneous noncommutative symmetric functions expand into the ribbon basis according to the following formula.  Let $\alpha$ be a strong composition.  Then

\begin{equation}{\label{eq:HtoR}}
H_{\alpha} = \sum_{\beta \succeq \alpha} R_{\beta}, \end{equation} where $\succeq$ is the refinement ordering on compositions. 

Set $\Delta_i=\Delta(\h(i,\tau_i))$
and $\Delta (\gamma)=(\Delta_1,\Delta_2,\cdots,\Delta_k).$ 
 Therefore the homogeneous expansion of the immaculates becomes  \begin{align*}
\I_{\alpha} & = \sum_{\gamma\in \SHF_{\mu}} \prod_{r=1
}^k \ \sign(\h(r,\c_r))\ H_{\Delta_r} \\
& = \sum_{\gamma\in \SHF_{\mu}} \Biggl( \prod_{r=1
}^k \ \sign(\h(r,\c_r)) \Biggr) H_{\Delta(\gamma)} \\
& = \sum_{\gamma\in \SHF_{\mu}} \Biggl( \prod_{r=1
}^k \ \sign(\h(r,\c_r)) \Biggr) \sum_{\beta \succeq \fl({\Delta(\gamma)})} R_{\beta}, \; \; \; \; \textrm{ by Equation~(\ref{eq:HtoR})}.
\end{align*}
Note that since $\alpha_r \ge r$ for all $1 \le r \le k$, it is not possible for all the cells in row $r$ to be covered by tunnel hooks originating in rows $1$ through $r-1$, so every tunnel hook starts within the \GBPR~diagram $D^{(0)}_{\alpha}$.  Therefore $\Delta(\h(r, \c_r)) > 0$ for all $1 \le r \le k$ for every tunnel hook covering $\gamma$.  Every indexing composition $\beta$ in the ribbon expansion can be thought of as a pair $(\Delta, S)$ where $\Delta=(\Delta_1, \Delta_2, \hdots , \Delta_k)$ is the indexing composition corresponding to a tunnel hook covering $\gamma$ and $S$ is the unique allowable flat coarsening subset such that $\beta=\Theta(\Delta,S)$.

We describe a sign-reversing involution on the pairs $(\Delta,S)$ with $S \not=\emptyset$, proving  that every term with less than $k$ parts cancels out.  For an arbitrary pair $(\Delta,S)$ corresponding to a term $R_{\Theta(\Delta,S)}$, let $i$ be the smallest element in $S$ and set $g(\Delta,S)=(f_i(\Delta),S)$.  The map $f_i$ is a sign-reversing involution (by Definition~\ref{D:involution} and Lemma~\ref{lem:signs}) and any coarsening set $S$ is an allowable flat coarsening set for any indexing composition of a tunnel hook covering of $\alpha$ since all parts are nonzero).  Therefore the pair $(f_i(\Delta),S)$ corresponds to the term $-R_{(\Theta(f_i(\Delta),S)}=-R_{\Theta(\Delta,S)}$ since $\Delta_i+\Delta_{i+1} =(f_i(\Delta))_{i+1}+  (f_i(\Delta))_i $.  Therefore the only terms appearing in the ribbon expansion are exactly the $R_{\beta}$ such that the length of the strong composition $\beta$ equals $k$, completing the proof.
\end{proof}

We next use \allowhooks~to provide a combinatorial proof of Campbell's ribbon expansion of immaculate functions~\cite{Cam17}.  To do this, we make use of the following lemma.

\begin{lem}{\label{lem:rect}}
Let $\gamma$ be a tunnel hook covering of the diagram of a composition $\alpha=(\alpha_1, \hdots , \alpha_k)$.  Let $j$ be the smallest value such that $\alpha_i=\alpha_k$ for all $i$ such that  $j \le i \le k$.  If $\h(r,\c_r)$ is a tunnel hook of $\gamma$ with $j \le r \le k$ and such that $\h(r,\c_r)$ includes at least one cell outside of the \GBPR~diagram $D_{\alpha}$, then either $\h(r,\c_r)$ consists of a single purple cell or there exists a tunnel hook $\h(s, \c_{s})$ in $\gamma$ with $r \le s \le k$ such that $\Delta(\h(s,\c_{s})) < 0$.
\end{lem}

\begin{proof}
Assume that $\h(r, \c_r)$ is a tunnel hook with $j \le r \le k$ containing at least one cell outside of the \GBPR~diagram $D_{\alpha}$.  If $\h(r,\c_r)$ consists of just one cell then we are done.  Therefore we may assume that $\h(r,\c_r)$ contains more than one cell.  We will show that there exists a tunnel hook $\h(s,\c_s)$ in $\gamma$ with $r \le s \le k$ such that $\Delta(\h(s,\c_s)) < 0$.  First, recall that tunnel hooks move in a northwestern (up and to the left) direction.  Since $\alpha_i=\alpha_k$ for all $j \le i \le k$, once a tunnel hook is inside the diagram $D_{\alpha}$ in row $j$ or higher, it cannot ever leave the diagram.  This means that in order to contain a cell outside of the diagram, the tunnel hook $\h(r, \c_r)$ must begin outside of the diagram.  Let $D^{(t)}_{\alpha/\nu^{(t)}} \setminus D_{\alpha}$ be the set of all cells introduced in rows $t$ and above during the construction of the first $t$ tunnel hooks of the tunnel hook covering.  Then, for all $t$ such that $1 \le t \le k$, all cells in $D^{(t)}_{\alpha/\nu^{(t)}} \setminus D_{\alpha}$ must either be purple or red.  If any cells in row $r$ are red and $\h(r, \c_r)$ terminates in row $r$, then $\Delta(\h(r,\c_r)) < 0$ and we are done.  Therefore assume that either no cells in row $r$ are red or $\h(r,\c_r)$ terminates in a row higher than row $r$.  If no cells in row $r$ are red, $\h(r,\c_r)$ must contain exactly one purple cell in row $r$, immediately to the right of a cell contained in $\alpha$. 
 This means $\h(r,\c_r)$ must terminate in a higher row since otherwise $\h(r, \c_r)$ would consist of only one cell.  Therefore we may now assume $\h(r,\c_r)$ terminates in a row higher than row $r$.  Then $\h(r,\c_r)$ includes at least one purple cell in row $r+1$ (the cell immediately above the initial cell for $\h(r,\c_r)$), which creates a red cell in row $r+1$ in Step 2(c) of Procedure~\ref{A:THC}. 

 We have now shown that any tunnel hook $\h(r,\c_i)$ such that $j \le r \le k$ originating outside the diagram of $\alpha$ consisting of more than one cell must result in at least one red cell in row $r+1$.  Now the tunnel hook $\h(r+1,\c_{r+1})$ originating in row $r+1$ includes the red cell in row $r+1$ and therefore can only be nonnegative if it introduces a red cell in row $r+2$.  Repeated iteration of this argument shows that once a red cell is introduced, there exists a positive integer $r \le s \le k$ such that $\Delta(\h(s,\c_s))<0$. 
\end{proof}

We are now ready to use tunnel hook coverings to prove Campbell's Rectangle Theorem.

\begin{theorem}{\label{T:Rectangle}}
Let $\alpha=(m^k)$ be a rectangle with $k$ rows each of length $m$.  Then 
$$\I_{(m^k)}=\sum_{\sigma \in S_n} \sign(\sigma) R_{(m-1+\sigma_1,m-2+\sigma_2, \hdots , m-k+\sigma_k)},$$ with the convention that $R_{\beta}$ vanishes if $\beta$ contains any nonpositive parts.
\end{theorem}

\begin{proof}
Let $\alpha=(m^k)$ be a rectangle with $k$ rows each of length $m$.  Letting $j=1$ in Lemma~\ref{lem:rect} implies that if $\gamma$ is a tunnel hook covering of $(m^k)$ such that $\Delta(\h(r, \c_r)) \ge 0$ for all $1 \le r \le k$, then every tunnel hook $\h(r,\c_r)$ in $\gamma$ either consists of a single purple cell outside the diagram (so that $\Delta(\h(r,\c_r))=0$) or is contained entirely within the rectangle $(m^k)$ (so that $\Delta(\h(r,\c_r)) >0$).  Since $H_0=1$, any occurrence of $H_0$ will simply disappear in the $H$-expansion of $\I_{\alpha}$.  

We claim that for the rectangle $(m^k)$, we cannot have two different tunnel hook coverings $\gamma \not= \gamma'$ such that $\fl(\Delta(\gamma)) = \fl(\Delta(\gamma'))$ and $(\Delta(\gamma))_i , (\Delta(\gamma'))_i\ge 0$ for $1 \le i \le k$.  To see this, assume $\fl(\Delta(\gamma)) = \fl(\Delta(\gamma'))$ for some $\gamma \not= \gamma'$ such that $(\Delta(\gamma))_i , (\Delta(\gamma'))_i\ge 0$ for all $1 \le i \le k$.  Let $j$ be the smallest positive integer such that $(\Delta(\gamma))_j \not= (\Delta(\gamma'))_j$.  If both $(\Delta(\gamma))_j >0$ and $(\Delta(\gamma'))_j >0$, then their flats must differ at the corresponding position.  One of $(\Delta(\gamma))_j$ or $(\Delta(\gamma'))_j$ must therefore be zero.  Assume without loss of generality that $(\Delta(\gamma))_j=0$.  Then the tunnel hook of $\gamma$ originating in row $j$ must begin with a cell outside of the diagram of $\alpha$.  But then the tunnel hook of $\gamma'$ originating in row $j$ must also begin with a cell outside of the diagram of $\alpha$, and therefore Lemma~\ref{lem:rect} implies that $(\Delta(\gamma'))_j=0$, contradicting the assumption that $(\Delta(\gamma))_j \not= (\Delta(\gamma'))_j$.  Therefore there is only one tunnel hook covering $\gamma$ such that $(\Delta(\gamma))_i \ge 0$ for all $1 \le i \le k$ with flattening $\fl(\Delta(\gamma))$.

We may therefore write the $H$-basis expansion of $\I_{(m^k)}$ using the flattenings $\fl(\Delta)$ as the subscripts, keeping track of the flattenings by associating the term $\fl(\Delta)$ with the pair $(\Delta,S)$, where $S$ is the unique allowable flat coarsening subset such that $\fl(\Delta) = \Theta(\Delta,S)$.  Since only one tunnel hook covering produces a given flattened composition, no entries cancel out in the $H$-basis expansion.  

To prove that the ribbon expansion satisfies Equation~(\ref{E:fullsize}), we first ignore any tunnel hook covering $\gamma$ such that $\Delta(\gamma)$ includes negative parts since $H_z=0$ when $z<0$.  Next, we construct a sign-reversing involution on pairs $(\Delta,S)$ where $\Delta$ is a weak composition obtained from a tunnel hook covering of $(m^k)$ and $S \not= \emptyset$ is a nontrivial allowable flat coarsening subset for $\Delta$.

Consider an arbitrary tunnel hook covering $\gamma$ of $\alpha$ described by tunnel cells $(\c_1, \c_2, \hdots , \c_i, \c_{i+1} , \hdots, \c_k)$ and corresponding indexing weak composition $\Delta (\gamma)=(\Delta_1, \Delta_2, \hdots , \Delta_k)$ such that $\Delta_i \ge 0$ for all $i$.  Let $R_{\beta}$ be a term in the ribbon expansion of $\I_{\alpha}$ with $\beta=\Theta(\Delta,S)$ such that $S = \{j_1, \hdots , j_s\}$ is a nontrivial allowable flat coarsening subset for $\Delta$.  Let $i$ be the index of the smallest element of $S$.  If $\c_i$ is not contained within the rectangle $(m^k)$, then $\Delta_i=0$ by Lemma~\ref{lem:rect}.   But if so, then $i-1 \in S$ (by the definition of allowable flat coarsening subset) which contradicts the assumption that $i$ is the smallest element of $S$.  Therefore $\c_i$ must be contained within the rectangle $(m^k)$ and we must have $\Delta_i > 0$.

Next, set $g(\Delta,S) = (f_i(\Delta),S)$, where $f_i(\Delta)$ is the sequence obtained by applying $f_i$ to the tunnel hook covering with associated sequence $\Delta$.  This map is sign-reversing by Lemma~\ref{lem:signs}.  Since $f_i$ is an involution and the set $S$ is unchanged, we must prove that $f_i(\Delta)$ is a weak composition and also that $S$ is an allowable flat coarsening subset for $f_i(\Delta)$.   If so, 
then every $\beta$ appearing as a subscript in the ribbon expansion of $\I_{(m^k)}$ satisfies $\beta=\Theta(\Delta,S)$ and is canceled out by the corresponding term $-R_{\beta}$ obtained from $\beta=\Theta(f_i(\Delta),S)$.  Therefore every $R_{\beta}$ whose indexing composition $\beta$ has length less than $k$ will be cancelled out in the ribbon expansion.

First note that Lemmas~\ref{T:perm} and~\ref{T:Mrj} imply that $(f_i(\Delta))_j=\Delta_j$ for $j \not= i,i+1$.  Therefore to see that $f_i(\Delta)$ is a weak composition, we must prove that $(f_i(\Delta))_i \ge 0$ and $(f_i(\Delta))_{i+1} \ge 0$.  

Recall that $\Delta_i >0$ since $i$ is the smallest element in the allowable flat coarsening subset $S$.  This means that the $i^{th}$ tunnel hook in $\gamma$ starts and ends within the diagram $(m^k)$.  Since the map $f_i$ changes the terminal cell but not the initial cell and tunnel hooks move north and west through the rectangle $(m^k)$, the $i^{th}$ tunnel hook in $f_i(\gamma)$ will also start and end within the diagram.  Therefore $(f_i(\Delta))_i >0$ we now only need to prove that $(f_i(\Delta))_{i+1} \ge 0$.  The only way for $(f_i(\Delta))_{i+1} < 0$ is if the $(i+1)^{th}$ tunnel hook of $f_i(\gamma)$ begins in a red cell.  But red cells can only be created when earlier tunnel hooks travel through purple or red cells.  Since all earlier tunnel hooks remain within the diagram of $(m^k)$, no red cells are created and the smallest possible value for $(f_i(\Delta))_{i+1}$ is $0$. 

Next, to prove that $S$ is an allowable flat coarsening subset for $f_i(\Delta)$, we must show that if $(f_i(\Delta))_r=0$, then $r-1 \in S$.  First consider the case that $r \not= i,i+1$.  Recall that $(f_i(\Delta))_r=\Delta_r$ for $r \not= i, i+1$.  If $(f_i(\Delta))_r=0$, then $\Delta_r=0$, and therefore $r-1 \in S$ since $S$ is an allowable flat coursening subset for $\Delta$.  Note that we have already shown $(f_i(\Delta))_i >0$, so the last case we need to consider is if $(f_i(\Delta))_{i+1} =0$.  But $i \in S$, so $S$ is an allowable flat coarsening subset for $f_i(\Delta)$ and the proof is complete. 
\end{proof}

We now combine the classes of compositions introduced in Theorems~\ref{thm:diagonal} and~\ref{T:Rectangle} to produce a larger class of compositions for which Equation~\eqref{E:fullsize} applies, therefore generalizing Campbell's results.

\begin{theorem}{\label{T:im2rib}}
Let $\alpha=(\alpha_1, \alpha_2, \hdots , \alpha_k)$ be a composition of positive integers for which there exists $J$ with $1 \le J \le k$ such that
$\alpha_{\ell} \ge \ell$ for $1 \le \ell \le J$ and $\alpha_{\ell} = J$ for $J+1 \le {\ell} \le k.$
Then 
$$\I_{\alpha}=\sum_{\sigma \in S_k} \sign(\sigma) R_{(\alpha_1-1+\sigma_1,\alpha_2-2+\sigma_2, \hdots , \alpha_k-k+\sigma_k)},$$ with the convention that $R_{\beta}$ vanishes if $\beta$ contains any nonpositive parts.
\end{theorem}

\begin{proof}
As in the proof of Theorem~\ref{T:Rectangle}, we must prove that if $R_{\beta}$ (where $\beta$ is a strong composition) is a term appearing in the ribbon expansion of $\I_{\alpha}$ (after cancellation) then the length of $\beta$ is $k$.  To see this, we again describe a sign-reversing involution on pairs $(\Delta,S)$ where $\Delta$ is a weak composition obtained from a tunnel hook covering $\gamma$ of $\alpha$ and $S \not= \emptyset$ is a nontrivial allowable flat coarsening subset for $\Delta$.

Let $g(\Delta,S)=(f_i(\Delta),S)$ where $i$ is the smallest element in $S$.  Again, it is enough to show that $f_i(\Delta)$ is a weak composition  and $S$ is an allowable flat coarsening subset for $f_i(\Delta)$.  Since $(f_i(\Delta))_j = \Delta_j$ for $j \not= i, i+1$ by Lemmas~\ref{T:perm} and~\ref{T:Mrj}, to see that $f_i(\Delta)$ is a weak composition, we only need to show $(f_i(\Delta))_i \ge 0$ and $(f_i(\Delta))_{i+1} \ge 0$.
  
First assume $i<J$.  Then $(f_i(\Delta))_i > 0$ and $(f_i(\Delta))_{i+1} >0$ since $\alpha_i \ge i$ and $\alpha_{i+1} \ge i+1$.  Next, consider the case $i=J$.  Then $\alpha_i \ge i$, so $(f_i(\Delta))_i >0$ and we must show $(f_i(\Delta))_{i+1} \ge 0$.  The only way $(f_i(\Delta))_{i+1} < 0$ is if there are red cells in row $i+1$ after the construction of the first $i$ tunnel hooks of $f_i(\gamma)$.  But since $i=J$ and all rows of $\alpha$ higher than row $i$ have length $J$, all the tunnel hooks in rows $1$ through $i$ remain completely inside the original diagram $\mathcal{D}_{\alpha}$.  Therefore there are no red cells anywhere in the diagram after the construction of the first $i$ tunnel hooks of $f_i(\gamma)$.  So $(f_i(\Delta))_{i+1} \ge 0$.

Finally, assume $i>J$.  Then $\alpha_i=J=\alpha_{i+1}$ and in fact $\alpha_r = J$ for all $r \ge i$.  Therefore Lemma~\ref{lem:rect} applies to all tunnel hooks of $f_i(\gamma)$ constructed in rows $i$ and above.  Since we assumed that $\Delta_r \ge 0$ for all $1 \le r \le k$, any tunnel hook $\h_r$ (where $\h_r=\h(r, \c_r)$) in $\gamma$ (with $r \ge i$) such that $\Delta_r >0$ must lie entirely inside the diagram of $\alpha$.  (Otherwise by Lemma~\ref{lem:rect}, there would be a value $\ell$ such that $\Delta_{\ell} <0$, a contradiction.)  In particular, since $\Delta_i > 0$, the tunnel hook $\h_i$ must lie entirely within the diagram of $\alpha$.  But this means that there are blue cells in row $i$ when the tunnel hook $h'_i$ of $f_i(\gamma)$ originating in row $i$ is constructed.  Therefore $(f_i(\Delta))_i >0$.  Since tunnel hooks move north and west, $\h'_i$ remains entirely inside the diagram.  Therefore no red cells are created during the construction of $\h'_i$.  Since no red cells are created during the constructions of the tunnel hooks in rows $1$ to $i-1$ either, there are no red cells in row $i+1$ of $f_i(\gamma)$ when the tunnel hook $\h'_{i+1}$ originating in row $i+1$ is constructed.  Therefore $(f_i(\Delta))_{i+1} \ge 0$.

Next, to show $S$ is an allowable flat coarsening subset for $f_i(\Delta)$, we must prove that if $(f_i(\Delta))_r=0$, then $r-1 \in S$.  Consider first the case that $r \not= i,i+1$.  Since $S$ is an allowable flat coarsening subset for $\Delta$ and $(f_i(\Delta))_r=\Delta_r$ for $r \not= i, i+1$, if $(f_i(\Delta))_r=0$, then $\Delta_r=0$, and therefore $r-1 \in S$.  Note that we have already shown $(f_i(\Delta))_i >0$, so the last case we need to consider is if $(f_i(\Delta))_{i+1} =0$.  But since $i \in S$, our proof is complete.  Therefore $S$ is an allowable flat coarsening subset for $f_i(\Delta)$.
\end{proof}

To see that Campbell's rectangles $(a^b, c^d)$ with $b \le c$ and $b \le a$ are contained in our set of compositions for which Theorem~\ref{T:im2rib}, let $\alpha_1=\alpha_2= \hdots = \alpha_b=a$ and $\alpha_{b+1}=\alpha_{b+2} = \hdots = \alpha_{b+d}=c$.  Set $J=c$ in Theorem~\ref{T:im2rib}.  For $1 \le \ell \le b$, we have $\alpha_{\ell} = a \ge b \ge \ell$.  For $b \le \ell \le c$, we have $\alpha_{\ell}=c \ge \ell$.  For $c+1 \le \ell \le k$, we have $\alpha_{\ell}=c$.  Therefore Campbell's rectangles satisfy the hypotheses of Theorem~\ref{T:im2rib} with $J=c$, but are certainly not the only compositions satisfying these hypothess.


The compositions appearing in Theorem~\ref{T:im2rib} are not the full set for which Equation~\eqref{E:fullsize} is true.  For example, $$\I_{1123}=R_{1123}-R_{1132}-R_{1213}+R_{1231}+R_{1312}-R_{1321}$$ but $(1,1,2,3)$ does not satisfy the conditions in Theorem~\ref{T:im2rib} since $\alpha_2 < 2$ but $\alpha_3 \not= \alpha_2$.  We, like Campbell, leave the full classification as an open problem.

\section*{Acknowledgements}
We would like to thank Mike Zabrocki and Aaron Lauve for their helpful comments and thoughtful suggestions with respect to this project. 


\nocite{*}

\bibliographystyle{amsplain}

\bibliography{IHbib}
\label{sec:biblio}

\end{document}